\documentclass[12pt,a4paper]{amsart}
\usepackage{geometry}
\usepackage{hyperref,color}
\allowdisplaybreaks
\usepackage{enumerate,pstricks,amssymb}

\newtheorem{theorem}{Theorem}
\newtheorem*{maintheorem*}{Main Theorem}
\newtheorem{lemma}[theorem]{Lemma}
\newtheorem{prop}[theorem]{Proposition}
\newtheorem{cor}[theorem]{Corollary}
\newtheorem{definition}[theorem]{Definition}
\theoremstyle{remark}
\newtheorem{remark}[theorem]{Remark}

\numberwithin{equation}{section}
\numberwithin{theorem}{section}

\newcommand{\lip}{\left\langle}
\newcommand{\rip}{\right\rangle}

\newcommand{\Der}[1]{\mathcal{D}(\Lie{#1})}
\newcommand{\SpinDer}[1]{\mathcal{R}(\Lie{#1})}
\newcommand{\SymDer}[1]{\mathcal{D}^{\mathrm{sym}}(\Lie{#1})}
\newcommand{\SkewDer}[1]{\mathcal{D}^{\mathrm{skew}}(\Lie{#1})}
\newcommand{\SymDerzero}[1]{\mathcal{D}^{\mathrm{sym}}_0(\Lie{#1})}
\newcommand{\SkewDerzero}[1]{\mathcal{D}^{\mathrm{skew}}_0(\Lie{#1})}
\newcommand{\Lie}[1]{\mathfrak{#1}}

\newcommand{\transp}{^{\mathsf{T}}}

\newcommand{\R}{\mathbb{R}}
\newcommand{\Z}{\mathbb{Z}}

\newcommand{\ad}{\operatorname{ad}}

\newcommand{\rspan}{\operatorname{span}}

\def \and {\quad\text{and}\quad}
\def\rist {\vert_}

\DeclareMathOperator{\height}{height}

\DeclareMathOperator{\End}{End}
\DeclareMathOperator{\Hom}{Hom}

\newcommand{\Epsilon}{\mathrm{E}}

\newcommand{\subsubsubsection}[1]{\noindent{\textit{#1}}\enspace}

\title[Derivations of Lie algebras]{On derivations of subalgebras \\ of real semisimple Lie algebras }

\author{Paolo Ciatti}
\address{
Universit\`a di Padova\\
Via Trieste 63\\
35121 Padova\\
ITALY}
\email {paolo.ciatti@unipd.it}

\author{Michael G.~Cowling}
\address{
University of New South Wales\\
UNSW Sydney 2052\\
AUSTRALIA}
\email {m.cowling@unsw.edu.au}

\thanks{Research supported by the Italian \emph{Ministero dell'Istruzione, dell'Universit\`a e della Ricerca} through PRIN 2011--12 \emph{``Variet\`a reali e complesse: geometria, topologia e analisi armonica''} and the GNAMPA project \emph{``Calcolo funzionale per operatori subellitici su variet\`a''}, and by the Australian Research Council through DP-140100531 ``Generalised conformal mappings''.
The first-named author thanks the University of New South Wales and the second-named author thanks the \emph{Universit\`a di Padova} for hospitality.  }


\begin{document}
\begin{abstract}
Let $\mathfrak{g}$ be a real semisimple Lie algebra with Iwasawa decomposition $\mathfrak{k} \oplus \mathfrak{a} \oplus \mathfrak{n}$.
We show that, except for some explicit exceptional cases, every derivation of the nilpotent subalgebra $\mathfrak{n}$ that preserves its restricted root space decomposition is of the form $\text{ad}( W)$, where $W \in \mathfrak{m}\oplus \mathfrak{a}$.
\end{abstract}
\maketitle
\section{Introduction}
Let $\Lie{k} \oplus \Lie{a} \oplus \Lie{n}$ be an Iwasawa decomposition of a real semisimple Lie algebra $\Lie{g}$ (here and later, $\oplus$ denotes a vector space direct sum; in general, the summands need not be Lie algebras), and let $\Lie{m}$ be the centraliser of $\Lie{a}$ in $\Lie{k}$.
We study the Lie algebra of derivations of the nilpotent subalgebra $\Lie{n}$ that preserve its restricted root space decomposition.
We show that every such derivation is of the form $\ad(W)$, where $W \in \Lie{m} \oplus \Lie{a}$, unless $\Lie{g}$ contains a simple summand of the form $\Lie{so}(n,1)$ or $\Lie{su}(n,1)$.
These derivations are known for real rank one simple Lie algebras.
Indeed, Kor\'anyi \cite{Ko} showed that in that case $\Lie{n}$ is an $H$-type Lie algebra, and the Lie algebra of derivations and the automorphism group of an $H$-type algebra were found by Riehm \cite{R} and Saal \cite{S}.

This paper is a step towards the classification of the derivations and automorphisms of $\Lie{n}$, which is interesting for a variety of reasons.
One reason is to find the derivations of (minimal) parabolic subalgebras of semisimple Lie algebras, which has been a lively field in recent years; see, for example, Chen \cite{Chen} and Wang and Yu \cite{WaYu}.
Every derivation of a parabolic subalgebra induces a derivation of its nilradical; if we can show that these are Lie multiplication by elements of the subalgebra, then we are well on the way to finding the derivations of the whole subalgebra.
Another reason is the problem of classification of nilpotent Lie algebras: in general this is an impossibly tedious matter, but one might hope to do better with algebras with lots of symmetry; to see whether this is viable, we need to understand some examples.

Next, to carry out harmonic analysis on the simply connected nilpotent Lie group associated to $\Lie{n}$, which has applications in diverse areas including theoretical physics and linear partial differential equations, it is important to understand its symmetries; see, for example, the study of Folland \cite{F}.

A fourth reason for studying the automorphisms of $\Lie{n}$ is the theory of quasiconformal mappings of ``Carnot groups''.
Indeed, as defined by Pansu \cite{Pu}, the derivative of a quasiconformal mapping of an Iwasawa $N$ group is an automorphism, and restrictions on the automorphisms give rise to restrictions on the quasiconformal mappings.
Further, it was shown by Yamaguchi  \cite{Y}, using the theory of Tanaka prolungations and the Borel--Bott--Weil theorem, and Cowling, De Mari, Kor\'anyi and Reimann \cite{CDMKR}, using more elementary arguments, that the space of ``multicontact mappings'', that is, mappings whose differentials preserve the simple root spaces, is finite-dimensional when all the  derivations that preserve the root spaces are of the form $\ad(\Lie{m} \oplus \Lie{a})$. 
We believe that the result presented here leads to the same conclusion in an even simpler way.
Indeed, unless $\Lie{n}$ has dimension $1$ or $2$, the Tanaka prolongation of $\Lie{n}$ through $\ad(\Lie{m}\oplus\Lie{a})$ is finite-dimensional; see Ottazzi and Warhurst    \cite{OtWa},  and this implies that multicontact mappings form a finite-dimensional Lie group.

It is also of interest to consider derivations that preserve the grading of $\Lie{n}$, that is, the subspaces $\sum_{\alpha} \Lie{g}_\alpha$ where we sum over all $\alpha$ of the same height, and to consider derivations of nilradicals of more general parabolic algebras; we will return to these questions in future work.

This paper is organized as follows. 
In Section~2 we analyse the derivations of an $H$-type algebra. 
We start by showing that every derivation is the sum of a symmetric derivation and a skew-symmetric derivation; we then describe symmetric and skew-symmetric derivations separately.

In Section~3, we consider real semisimple Lie algebras. 
First, we reduce matters to the case of simple Lie algebras, and then we show that these contain various $H$-type algebras. 
We also see how the geometry of root systems is reflected in the structure of various subalgebras of $\Lie{g}$. 
Most of the ideas behind this section may be found in Ciatti \cite{C1, C2, C3, C4, C5}.

In Section~4, we consider the grading of a semisimple Lie algebra associated to a choice of positive roots,  and grading preserving derivations of $\Lie{g}$, of $\Lie{m} \oplus \Lie{a} \oplus \Lie{n}$ and of $\Lie{n}$.
We find a simple Lie algebraic criterion for a skew-symmetric grading preserving derivation of $\Lie{n}$ to extend to a derivation of $\Lie{g}$; this extended derivation is not only grading preserving but also root space preserving.

Finally, in Section~5 we apply the results of Sections~2 and~3 to the study of the derivations of $\Lie{n}$ that preserve the root space decomposition.
These are sums of symmetric derivations and skew-symmetric derivations. 
The main idea is to show that our assertion is true when the real rank of $\Lie{g}$ is $1$ or $2$, and then apply this result to the rank two subalgebras of a general simple Lie algebra, deducing from these the full result.

\begin{maintheorem*}
If no simple summand of $\Lie{g}$ is isomorphic to $\Lie{so}(n, 1)$ or $\Lie{su}(n,1)$ for any $n$, then all the derivations of $\Lie{n}$ that preserve the root spaces are of the form $\ad(W)$, where $W\in \Lie{m} \oplus \Lie{a}$.
Otherwise, there are derivations of $\Lie{n}$ that preserve the root spaces that do not arise in this way.
\end{maintheorem*}

\section{Derivations of an $H$-type Lie algebra}

In this section, we first define $H$-type Lie algebras, which arose in the work of Kaplan \cite{Kap}, and then describe their derivations.
These are always the sum of a symmetric derivation and a skew-symmetric derivation.
In Corollary~\ref{cor:2.6}, skew-symmetric derivations are decomposed as the sum of two components, one of which is trivial on the centre.
The symmetric derivations are classified in Corollary~\ref{cor:2.8} by a diagonalization process.

Let $\Lie{h}$ be a two-step nilpotent Lie algebra, endowed with an inner product $\lip \cdot, \cdot \rip$.
We denote by $\Lie{z}$ the centre of $\Lie{h}$ and by $\Lie{v}$ the orthogonal complement of $\Lie{z}$;
given a subspace $\Lie{s}$ of $\Lie{h}$, we write $I_{\Lie{s}}$ for the identity map on $\Lie{s}$.
Then
\[
\Lie{g} = \Lie{v} \oplus \Lie{z}.
\]
For each $Z$ in $\Lie{z}$, we define $J_Z$ in $\End(\Lie{v})$ by
\begin{equation} \label{eq:defJZ}
\lip J_Z X, Y \rip = \lip Z, [X,Y] \rip
\quad\forall X, Y \in \Lie{v}.
\end{equation}
Then $J_Z$ is trivially skew-symmetric, that is, $J_Z\transp  = -J_Z$, where ${}\transp $ denotes the transpose relative to the inner product.
We say that $\Lie{h}$ is of Heisenberg type, or just $H$-type, when
\begin{equation}\label{eq:Jsquare}
J_Z^2 = -\| Z \|^2 I_{\Lie{v}}
\end{equation}
for all $Z \in \Lie{z}$.
Equivalently, for each $X \in \Lie{v}$ of length $1$, the map $\ad(X)$ is an isometry from $\ker(\ad(X))^\perp$ onto $\Lie{z}$. 
For the rest of this section, we assume that $\Lie{h}$ is an $H$-type algebra.

By polarization, \eqref{eq:Jsquare} implies that
\begin{equation}\label{eq:cliff}
J_Z J_{Z'} + J_{Z'} J_Z = - 2 \lip Z, Z' \rip I_{\Lie{v}}
\qquad \forall Z, Z' \in \Lie{z} .
\end{equation}
Thus the $J_Z$ generate a Clifford algebra.

Recall that a derivation of a Lie algebra $\Lie{h}$ is a linear endomorphism $D: \Lie{h} \to \Lie{h}$ such that 
\[
D([X,Y]) = [DX, Y] + [X, DY] 
\quad\forall X, Y \in \Lie{h};
\]
every derivation of a Lie algebra automatically preserves the centre.
We say that a linear endomorphism $D$ of $\Lie{h}$ \emph{preserves the grading} if $D(\Lie{v}) \subseteq \Lie{v}$ and $D(\Lie{z}) \subseteq \Lie{z}$, and write $\Der{h}$ for the Lie algebra of all grading preserving derivations of $\Lie{h}$.
We denote by $\SymDer{h}$ the subspace of $\Der{h}$ of all symmetric derivations and by $\SkewDer{h}$ the Lie subalgebra of all skew-symmetric derivations.
We also write $\SymDerzero{h}$ and $\SkewDerzero{h}$ for the subspaces of these spaces of derivations that vanish on $\Lie{z}$.

\begin{prop}\label{prop:2.1}
Let $D$ be a grading preserving linear endomorphism of $\Lie{h}$.
Then $D$ is a derivation if and only if
\begin{equation}\label{eq:2.4}
J_{D\transp  Z} = D\transp  J_Z + J_Z D
\quad\forall Z \in \Lie{z}.
\end{equation}
Suppose moreover that $D \rist{\Lie{z}} = 0$.
If $D$ is skew-symmetric, then $D$ is a derivation if and only if $D$ commutes with all the $J_Z$ and if $D$ is symmetric, then $D$ is a derivation if and only if $D$ anticommutes with all the $J_Z$.
\end{prop}

\begin{proof}
From \eqref{eq:defJZ}, it follows that $D$ is a derivation if and only if, for all $Z$ in $\Lie{z}$ and $X$, $Y$ in $\Lie{v}$,
\begin{align*}
\lip J_{D\transp  Z} X, Y \rip 
&= \lip D\transp  Z, [ X, Y ] \rip
= \lip Z, D [ X, Y ] \rip
\\&= \lip Z, [ D X, Y ] \rip + \lip Z, [ X, D Y ] \rip 
\\&= \lip J_Z {D X}, Y \rip + \lip D\transp  J_Z X, Y \rip ,
\end{align*}
proving the result.
\end{proof}

The next result is known, but we give a proof for completeness.

\begin{lemma}[{Riehm \cite{R}}]\label{lem2.3}
For every pair of orthogonal vectors $Z'$ and $Z''$ in $\Lie{z}$, the grading preserving linear map $\Phi_{Z' Z''}$, defined by
\[
\Phi_{Z' Z''} (X+Z) = J_{Z'} J_{Z''} X + 2\lip Z', Z \rip Z'' - 2\lip Z'', Z \rip Z' 
\]
for all $Z \in \Lie{z}$ and all $X \in \Lie{v}$, is a skew-symmetric derivation of $\Lie{h}$.
\end{lemma}

\begin{proof}
It is evident that $\Phi_{Z' Z''}$ is skew-symmetric.
By Proposition \ref{prop:2.1}, it suffices to show that
\begin{equation*}
 J_{\Phi_{Z' Z''} (Z)} X = \Phi_{Z' Z''} J_Z X - J_Z \Phi_{Z' Z''} X
\quad\forall X \in \Lie{v}.
\end{equation*}
We consider the right-hand side of the equation, and use \eqref{eq:cliff}:
\begin{align*}
J_{Z'}J_{Z''} J_Z X - J_Z J_{Z'}J_{Z''}X 
& = - 2 \lip Z ,{Z''}\rip J_{Z'}X - J_{Z'} J_ZJ_{Z''}X
\\
&\qquad+  2 \lip Z ,{Z'}\rip J_{Z''}X  + J_{Z'} J_ZJ_{Z''}X
\\
&= 2\lip Z', Z \rip J_{Z''}X - 2\lip Z'', Z \rip J_{Z'}X 
\\
&= J_{2\lip Z', Z \rip Z'' - 2\lip Z'', Z \rip Z'}X
\\
&= J_{\Phi_{Z' Z''} (Z)}X ,
\end{align*}
as required.
\end{proof}

We define $\SpinDer{h}$ to be the vector subspace of $\Der{h}$ of all grading preserving derivations of $\Lie{h}$ spanned by the $\Phi_{Z'Z''}$.
As observed by Riehm \cite{R}, the subspace $\SpinDer{h}$ is a \emph{subalgebra} of $\Der{h}$.
To see this, we take an orthonormal basis $\{ Z_1, \cdots, Z_m \}$ for $\Lie{z}$, and write $\Phi_{ij}$ in place of $\Phi_{Z_iZ_j}$.
Since $\Phi_{Z'Z''}$ depends linearly on $Z'$ and on $Z''$, every element of $\SpinDer{h}$ is a linear combination of the $\Phi_{ij}$. 
Moreover,
\begin{equation*}
\Phi_{ij} \Phi_{kl} - \Phi_{kl} \Phi_{ij} =
\begin{cases}
0              &\text{if $\{i,j\} \cap \{k,l\} = \varnothing$,} \\
2\Phi_{jl}  &\text{if $i=k$,}
\end{cases}
\end{equation*}
which shows that $\SpinDer{\Lie{h}}$ is closed under taking commutators.
We omit the proof of these commutation relations, as we do not need this result.

\begin{cor}\label{cor:2.4}
Suppose that $D$ is a grading preserving derivation of $\Lie{h}$.
Then we may write $D$ as $D_0 + D_1$, where $D_0 \in \Der{h}$ and $D_0\rist{\Lie{z}}$ is symmetric, and $D_1 \in \SpinDer{h}$.
\end{cor}

\begin{proof}
The skew-symmetric part of the restriction $D|_\Lie{z}$ of $D$ to $\Lie{z}$ decomposes as a linear combination of the $\Phi_{ij}|_\Lie{z}$ defined above; we take $D_1$ to be the same linear combination of the $\Phi_{ij}$, and $D_0$ to be $D - D_1$.
The result follows immediately.
\end{proof}

\begin{cor}\label{cor:2.5a}
Suppose that $D \in \Der{h}$ and $D\rist{\Lie{z}}$ is symmetric.
Then $D\transp \in \Der{h}$.
\end{cor}

\begin{proof}
Since $D\rist{\Lie{z}}$ is symmetric, it is diagonalisable.
Take an eigenvector $Z$ in $\Lie{z}$ with eigenvalue $2\mu$.
By Proposition \ref{prop:2.1},
\[
2\mu J_Z = J_{DZ}  = J_{D\transp  Z} = D\transp  J_Z + J_Z D,
\]
whence multiplication on both sides by $J_Z$ gives
\[
-2\mu|Z|^2 J_Z =  -|Z|^2 J_ZD\transp  - |Z|^2 D J_Z,
\]
and
\[
J_{DZ} =  D J_Z +  J_ZD\transp .
\]
This holds for all eigenvectors $Z$ of $D$, and so for all $Z \in \Lie{z}$ by linearity,  so $D\transp $ is a derivation, again by Proposition \ref{prop:2.1}. 
\end{proof}

\begin{cor}\label{cor:2.5b}
Suppose that $D$ is a grading preserving endomorphism of $\Lie{h}$.
Then $D \in \Der{h}$  if and only if $D\transp \in \Der{h}$.
Hence if $D \in \Der{h}$, then we may write $D$ as $D^a+D^s$, where $D^s \in \SymDer{h}$ and $D^a\in \SkewDer{h}$.
\end{cor}

\begin{proof}
For the first part, it suffices to suppose that $D \in \Der{h}$ and show that $D\transp \in \Der{h}$.
In light of Corollary \ref{cor:2.4}, by subtracting off an element of $\SpinDer{\Lie{h}}$ if necessary, we may assume that $D\rist{\Lie{z}}$ is symmetric.
It follows that $D\transp  \in \Der{h}$, as required. 

For the second part of the corollary, take
\[
D^s = \frac 12 (D+D\transp )
\and
D^a = \frac 12 (D-D\transp );
\]
the conclusion is obvious.
\end{proof}

Hence, to describe the elements of $\Der{h}$, we can study symmetric and skew-symmetric derivations separately.
First we consider the skew-symmetric derivations.

\begin{cor}\label{cor:2.6}
Each $D$ in $\SkewDer{h}$ decomposes as a sum $D_0 +R$, where $D_0 \in \SkewDerzero{h}$ and $R \in \SpinDer{h}$.
In particular, $D_0 \rist{\Lie{v}}$ commutes with all the maps $J_Z$.
\end{cor}

\begin{proof}
This is a consequence of  Corollary \ref{cor:2.4}  and  Proposition \ref{prop:2.1}.
\end{proof}

Now we consider  a symmetric derivation $D$, which is diagonalizable with real eigenvalues.
Since $D$ preserves $\Lie{v}$ and $\Lie{z}$, these spaces decompose into eigenspaces. 
We write $\Lie{v}_{\lambda}$ for the eigensubspace of $\Lie{v}$ associated to the eigenvalue $\lambda$ and, given a subspace $\Lie{s}$ of $\Lie{h}$, we write $P_{\Lie{s}}$ for the orthogonal projection of $\Lie{h}$ onto $\Lie{s}$.

\begin{prop}\label{prop:2.7} Suppose that $D\in \SymDer{h}$.
Then $D\rist{\Lie{z}} = 2\mu I_{\Lie{z}}$ for some $\mu$ in $\R$.
Moreover, if $X \in \Lie{v}_\lambda$, then
\begin{equation}\label{eq3.2}
D J_Z X = ( 2\mu - \lambda) J_Z X
\and
D J_Z J_{Z'} X = \lambda J_Z J_{Z'} X
\end{equation}
for all $Z$ and $Z'$ in $\Lie{z}$.
\end{prop}

\begin{proof}
Fix an orthonormal basis $\{ Z_1, \dots, Z_m \}$ of $\Lie{z}$ such that $D Z_i = 2\mu_i Z_i$ when $i = 1, \dots, m$, where each $\mu_i$ in $\mathbb R$.
From \eqref{eq:2.4}, it follows that
\[
D J_{Z_i} X = J_{D Z_i} X - J_{Z_i} D X = (2\mu_i - \lambda) J_{Z_i} X
\]
when $i = 1, \dots, m$, and the first formula of \eqref{eq3.2} is established, and similarly,
\begin{equation}\label{eq:2.6}
D J_{Z_i} J_{Z_k} X = (2\mu_i - 2\mu_k + \lambda) J_{Z_i} J_{Z_k} X
\end{equation}
when $i, k = 1, \dots, m$.

If $\dim(\Lie{z}) = 1$, then $D\rist{\Lie{z}} = 2\mu I_{\Lie{z}}$ for some $\mu$ in $\R$ and the second formula of \eqref{eq3.2} is trivial, so we suppose henceforth that $\dim(\Lie{z}) > 1$.
By interchanging $i$ and $k$ in \eqref{eq:2.6}, we see that
\[
D J_{Z_k} J_{Z_i} X = ( 2\mu_k - 2\mu_i  + \lambda) J_{Z_k} J_{Z_i} X ,
\]
which yields
\[
D J_{Z_i} J_{Z_k} X = (2\mu_k - 2\mu_i  + \lambda) J_{Z_i} J_{Z_k} X ,
\]
when $i \neq k$, since $J_{Z_i} J_{Z_k} = - J_{Z_k} J_{Z_i}$ by \eqref{eq:cliff}.
This equality, compared with \eqref{eq:2.6}, shows that $\mu_i = \mu_k$, and the lemma follows.
\end{proof}

\begin{cor}\label{cor:2.8}
Let $D$ be a derivation in $\SymDer{h}$.
Denote by $2\mu$ the eigenvalue of $D$ on $\Lie{z}$, and by $\{ \lambda_1, \dots, \lambda_r \}$ the distinct eigenvalues of $D$ on $\Lie{v}$, listed in decreasing order, and by $\Lie{v}_i$ the corresponding eigenspaces.
Then $\lambda_i + \lambda_{r+1-i} = 2\mu$, and we may write
\begin{equation*}
\begin{aligned}
D
&= \mu \bigl( 2 P_{\Lie{z}} +  P_{\Lie{v}} \bigr) +
\sum_{i = 1}^{\lfloor r/2 \rfloor} \left( \lambda_i - \mu \right) \bigl( P_{\Lie{v}_i} - P_{\Lie{v}_{r+1-i}}  \bigr);
\end{aligned}
\end{equation*}
all the maps $\bigl( P_{\Lie{v}_i} - P_{\Lie{v}_{r+1-i}}  \bigr)$ and $2 P_{\Lie{z}} + P_{\Lie{v}} $ are derivations.
\end{cor}

\begin{proof}
This follows from Propositions \ref{prop:2.7} and \ref{prop:2.1}.
\end{proof}

\section{ Structure of semisimple Lie algebras}

In this section, we describe the restricted root structure and the standard Iwasawa and Bruhat decompositions of a semisimple Lie algebra.
Then we exhibit a number of $H$-type subalgebras of the Iwasawa $\Lie{n}$ subalgebra. 
Next, we analyse the structure of $\Lie{g}$ and $\Lie{n}$ in more detail.

\subsection{Semisimple Lie algebras of the noncompact type.}
Take a real semisimple Lie algebra $\Lie{g}$ with Killing form $B$ and Cartan involution $\theta$, and let $\Lie{k} \oplus \Lie{p}$ be the corresponding Cartan decomposition of $\Lie{g}$.
Fix a maximal subalgebra $\Lie{a}$ of $\Lie{p}$; its dimension is known as the real rank of $\Lie{g}$.
Given an element $\alpha$ of $\Hom(\Lie{a}, \R)$, we define the (possibly trivial) subspace $\Lie{g}_{\alpha}$ of $\Lie{g}_{\alpha}$ by 
\[
\Lie{g}_{\alpha} = \{ X \in \Lie{g} : [ H, X ] = \alpha (H) X, \ \forall H \in \Lie{a} \} .
\]
Then $\alpha$ is said to be a restricted root if $\alpha \neq 0$ and $\Lie{g}_{\alpha} \neq \{0\}$.
We denote by $\Sigma$ the restricted root system, that is, the set of all restricted roots. 
Note that $[\Lie{g_\alpha}, \Lie{g}_\beta] \subseteq \Lie{g}_{\alpha+\beta}$ for all  $\alpha, \beta \in \Hom(\Lie{a}, \R)$, because $\ad(H)$ is a derivation for each $H \in \Lie{a}$. 
Hence if $\alpha$ and $\beta$ are roots, then $\alpha+\beta$ is also a root, unless $\alpha+\beta = 0$ or $[\Lie{g_\alpha}, \Lie{g}_\beta] =\{0\}$.
Since $\Lie{a}$ is $\theta$-invariant, so is $\Lie{g}_0$, and it follows easily that $\Lie{g}_0 = \Lie{m} \oplus \Lie{a}$, where $\Lie{m} = \Lie{g}_0 \cap \Lie{k}$.
Then
\[
\Lie{g} = \Lie{m} \oplus \Lie{a} \oplus \sum_{\alpha \in \Sigma} \Lie{g}_{\alpha} .
\]
Henceforth, in this paper, unless stated explicitly otherwise, we write rank and root rather than real rank and restricted root for brevity; this should not create any confusion.

We recall that $\Sigma$ is said to be decomposable if $\Sigma = \Sigma_1 \cup \Sigma_2$, where $\Sigma_1$ and $\Sigma_2$ are disjoint nontrivial subsets of $\Sigma$ and $\lip \gamma, \delta \rip = 0$ for all $\gamma \in \Sigma_1$ and all $\delta \in \Sigma_2$, and indecomposable otherwise.
It is standard (see, for instance, Helgason \cite{H} or  Knapp \cite{Kn}) that $\Sigma$ is indecomposable if and only if $\Lie{g}$ is simple, that is, cannot be written as a direct sum of nontrivial pairwise commuting ideals.
We recall also that $\Sigma$ is said to be reduced if the only multiples of a root $\gamma$ that also lie in $\Sigma$ are $\pm \gamma$.

A Weyl chamber is a maximal open subset of $\Lie{a}$ in which no root vanishes.
We choose one of these, $C$ say, and say that a root $\gamma$ is positive, and write $\gamma \in \Sigma^+$, when $\gamma(H) >0$ for all $H \in C$.
Then $\Sigma^+$ is closed under addition and $\Sigma = \Sigma^+ \cup (-\Sigma^+)$.
We write $\Delta$ for the smallest subset of $\Sigma^+$ such that the boundary of $C$ is a subset of the set $\bigcup_{\alpha \in \Delta} \{ H \in \Lie{a} : \alpha(H) = 0\}$; the roots in $\Delta$ are called \emph{simple}.
Set 
\[
\Lie{n} = \sum_{\alpha \in \Sigma^+} \Lie{g}_{\alpha} .
\]
Then we obtain the Bruhat decomposition of $\Lie{g}$, namely,
\[
\Lie{g} = \theta \Lie{n} \oplus \Lie{m} \oplus \Lie{a} \oplus \Lie{n} .
\]

Each root $\gamma$ in $\Sigma^+$ may be written uniquely as a sum $\sum_{\alpha\in\Delta} n_\alpha \alpha$, where each $n_\alpha$ is a nonnegative integer. 
The positive integer $\sum_{\alpha\in\Delta} n_\alpha$ is called the height of $\gamma$, and written $\height(\gamma)$. 
Clearly the height of a simple root is $1$, and moreover 
\[
\height(\gamma+\delta) = \height(\gamma) + \height(\delta)
\]
for all $\gamma, \delta \in \Sigma^+$ such that $\gamma+\delta \in \Sigma^+$.
Then $\Lie{n}$ is graded by height; more precisely, we may write $\Lie{n} = \sum_{h \in \Z^+} \Lie{g}_h$, where $[\Lie{g}_h, \Lie{g}_{k} ] \subseteq \Lie{g}_{h+k}$.

\subsection{Reduction to the simple case.}
Our first simplification is a reduction of the problem to the case of the Iwasawa $\Lie{n}$ subalgebra of a simple Lie algebra $\Lie{g}$.

\begin{prop}\label{prop:3.1}
Suppose that $\Lie{g} = \Lie{g}^1 \oplus \Lie{g}^2 \oplus \dots \oplus \Lie{g}^J$, where $J>1$ and each $\Lie{g}^j$ is a nontrivial simple ideal, and that $\Lie{n} = \Lie{n}^1 \oplus \Lie{n}^2 \oplus \dots \oplus \Lie{n}^J$ is the corresponding decomposition of $\Lie{n}$ into subalgebras.
Then 
\[
\Der{\Lie{n}} = \sum_{j=1}^J \Der{\Lie{n}^j}.
\]
\end{prop}

\begin{remark}
This is to be interpreted in the sense that each root space preserving derivation of $\Lie{n}$ preserves each of the subalgebras $\Lie{n}^j$, and the restriction to each subalgebra is a root space preserving derivation thereof, and  vice versa.

Some of the simple summands $\Lie{g}^j$ may be compact. 
In this case, the corresponding space $\Lie{n}^j$ is $\{0\}$; we define $\Der{\{ \mathrm{0} \}} = \{0\}$.
\end{remark}

\begin{proof}
Since $D$ in $\Der{\Lie{n}}$ preserves the root spaces, it preserves each $\Lie{g}_\alpha$ and hence each $\Lie{n}^j$.
So one direction of the assertion is proved.
The other is obvious.	
\end{proof}	

\begin{remark}
If we replace the root space preserving assumption by a grading preserving assumption, and add the hypothesis that no summand is isomorphic to $\Lie{so}(n,1)$ for any $n$, then the result still holds.
Indeed, when there is no abelian summand, $\Lie{n}$ is ``totally nonabelian'' in the language of Cowling and Ottazzi \cite{CoOt}, and the conclusion follows from \cite[Corollary 2.4]{CoOt}.
 \end{remark}

\subsection{The simple case}
In light of Proposition \ref{prop:3.1}, we may and shall assume that $\Lie{g}$ is simple in the rest of this paper.

Two observations underpin our approach to the study of derivations.
First, derivations are local, in the sense that if $D$ is a root space preserving linear endomorphism of $\Lie{n}$, then linearity implies that $D$ is a derivation if and only if 
\[
D[X,Y] = [DX,Y] + [X,DY] 
\qquad\forall X \in \Lie{g}_\gamma \quad\forall Y \in \Lie{g}_\delta,
\]
as $\gamma$ and $\delta$ range over $\Sigma^+$.
This identity holds trivially if $\gamma+\delta$ is not a root, for then both sides are $0$.
If $\gamma+\delta$ is a root, then the subalgebra $\Lie{n}^{\{\gamma, \delta\}}$, defined by
\[
\Lie{n}^{\{\gamma, \delta\}} = \sum_{\epsilon \in \Sigma^+ \cap (\R\gamma+\R\delta)} \Lie{g}_\epsilon,
\]
is the Iwasawa $\Lie{n}$ subalgebra of a simple subalgebra of $\Lie{g}$, whose rank is $1$ if $\gamma=\delta$ and $2$ otherwise.
Then we can understand $D$ provided we understand its restriction to Iwasawa $\Lie{n}$ algebras of simple Lie algebras of rank one and rank two.

The second observation is that $\Lie{n}$ may be equipped with a natural inner product so that, in the rank one case, $\Lie{n}$ itself is an $H$-type algebra, while in the rank two cases, $\Lie{n}$ has many $H$-type subalgebras.
We will use what we know about the derivations of $H$-type algebras, but first we need to find $H$-type subalgebras of $\Lie{n}$.

\subsection{Subalgebras of $\Lie{n}$ of $H$-type.}
If $c > 0$, then the symmetric bilinear form $\lip \cdot, \cdot \rip$ on $\Lie{g}$, given by
\begin{equation}\label{eq:3.1}
\lip X, Y \rip = - c B(X, \theta Y) ,
\end{equation}
is an inner product, which induces an inner product on the dual of $\Lie{a}$, also written $\lip \cdot, \cdot \rip$; we denote the corresponding norms by $\Vert \cdot \Vert$. 
We fix $c$ so that the length of the longest roots is $\sqrt{2}$.
In the vector space decomposition $\Lie{m} \oplus \Lie{a} \oplus \sum_{\alpha \in \Sigma} \Lie{g}_\alpha $ of $\Lie{g}$, the distinct summands  are orthogonal.

Now the Killing form satisfies the well-known identity 
\[
B( [Z,X],Y) + B( X, [Z, Y]) = 0 
\quad\forall X,Y,Z \in \Lie{g}, 
\]
and so $\ad(Y)\transp  = - \ad(\theta Y)$, that is,
\begin{equation}\label{eq:3.2}
\lip X, [ Y, Z ] \rip = - \lip [ \theta Y , X], Z \rip
\quad \forall X,Y,Z \in \Lie{g}.
\end{equation}
If $\gamma \in \Sigma$ and $X,Y \in \Lie{g}_\gamma$, then $[\theta X, Y] \in \Lie{g}_0$. 
Further, for all $H \in \Lie{a}$,
\begin{equation}\label{eq:theta-X-Y}
\lip H, [\theta X, Y] \rip
= - \lip [X, H], Y \rip
= \lip [H, X], Y \rip
= \gamma(H) \lip X, Y\rip. 
\end{equation}
On the one hand, if $X \perp Y$, then $\lip H, [\theta X, Y] \rip = 0$, and so $[\theta X, Y] \in \Lie{a}^\perp$, whence $[\theta X, Y] \in \Lie{m}$.
On the other hand, $\theta[\theta X,X] = - [\theta X, X]$, so $[\theta X, X] \in \Lie{a}$. 
We write $H_\gamma$ for the unique element of $\Lie{a}$ such that $\delta(H_\gamma) = \lip\delta,\gamma\rip$ for all $\delta \in \Hom(\Lie{a}, \R)$, or equivalently for all $\delta \in \Sigma$.
Now \eqref{eq:theta-X-Y} implies that 
\begin{equation}\label{eq:3.3}
\delta( [\theta X, X] ) = \lip \delta, \gamma \rip \Vert X\Vert^2
\and
[\theta X, X] = \Vert X\Vert^2 H_\gamma
\end{equation}
 for all $X \in \Lie{g}_\gamma$.
For future purposes, note that 
\begin{equation}\label{eq:decomp-of-Lie-a}
 \Lie{a} = \sum_{\alpha \in \Sigma^+} \R H_\alpha;
\end{equation}
in general, this sum is not direct.

Our next results allow us to find various subalgebras of $\Lie{n}$ that are $H$-type algebras, or nearly so.

\begin{lemma}\label{lemH0}
Suppose that $\gamma$, $\delta$, and $\gamma+\delta$ are positive roots.
For all $Z$ in $\Lie{g}_{\gamma+\delta}$, we define the linear operator $J_Z$ on $\Lie{g}_\gamma\oplus \Lie{g}_\delta$ by
\begin{equation}\label{defK}
J_Z = \ad(Z)  \circ \theta.
\end{equation}
Then $J_Z$ maps $\Lie{g}_\gamma$ into $\Lie{g}_\delta$ and $\Lie{g}_\delta$ into $\Lie{g}_\gamma$; further
\begin{equation}\label{eq:K1}
\lip J_Z X, Y \rip = \lip Z, [X,Y] \rip
\quad\forall X, Y \in \Lie{g}_\gamma\oplus \Lie{g}_\delta.
\end{equation}
\end{lemma}

\begin{proof}
The mapping properties of $J_Z$ are consequences of the orthogonality of distinct root spaces, while \eqref{eq:K1} follows from the definition of $J_Z$ and \eqref{eq:3.2}.
\end{proof}
 
 \begin{lemma}\label{lemH}
Suppose that $\gamma$, $\delta$, and $\gamma+\delta$ are positive roots, and that $J_Z$ is defined as in Lemma \ref{lemH0}.
Suppose also that neither $\gamma+2\delta$ nor $2\gamma+\delta$ is a root.
Then
\begin{equation}\label{[X,KX]}
[X,J_Z X] = \lip \gamma+\delta,\gamma \rip \|X\|^2 Z
\quad\forall X \in \Lie{g}_\gamma \oplus \Lie{g}_\delta
\end{equation}
and
\begin{equation}\label{K2}
J_Z ^2 X = - \lip \gamma+\delta,\gamma \rip \|Z\|^2 X
\quad\forall X \in \Lie{g}_\gamma\oplus \Lie{g}_\delta,
\end{equation}
Thus if $Z \neq 0$, then $J_Z$ is a linear isomorphism of $\Lie{g}_\gamma \oplus \Lie{g}_\delta$ that exchanges $\Lie{g}_\delta$ and $\Lie{g}_\gamma$.
Moreover, if $\gamma = \delta$, then $\Lie{g}_{\gamma} \oplus \Lie{g}_{2\gamma}$ is an $H$-type algebra, while if neither $2\gamma$ nor $2\delta$ is a root, then $\Lie{g}_{\gamma}\oplus \Lie{g}_{\delta} \oplus \Lie{g}_{\gamma+\delta}$ is an $H$-type algebra.
\end{lemma}
 
\begin{proof}
When $2\gamma+\delta$ is not a root, $[X,Z] = 0$ for all $Z$ in $\Lie{g}_{\gamma+\delta}$ and all $X$ in $\Lie{g}_\gamma$.
Hence, from the Jacobi identity and \eqref{eq:3.3},
\begin{align*}
[X, J_Z X]
&=[X, [Z,\theta X]]\\
&= [[X,Z], \theta X] + [Z, [X,\theta X]] \\
&=(\gamma+\delta)([\theta X,X]) Z\\
&=\lip \gamma+\delta,\gamma \rip \|X\|^2 Z,
\end{align*}
and similarly,
\begin{align*}
J_Z ( J_Z X)
&= [Z, \theta[Z,\theta X]]
\\
&= [Z, [\theta Z, X]]
\\
&= [X, [\theta Z, Z]]+[\theta Z, [Z, X]]
\\
&=  - \lip \gamma+\delta,\gamma \rip \|Z\|^2 X.
\end{align*}
By exchanging the role of $\gamma$ and $\delta$ in the last two formulae, we see that
\[
[Y, J_Z Y]
=\lip \gamma+\delta,\delta \rip \|Y\|^2 Z
\]
and
\[
J_Z ( J_Z Y)
= -\lip \gamma+\delta,\delta \rip \|Z\|^2 Y
\]
for all $Z \in \Lie{g}_{\gamma+\delta}$ and all $Y \in \Lie{g}_\delta$ when $\gamma+2\delta$ is not a root.
Hence \eqref{[X,KX]} and \eqref{K2} are proved, and $J_Z$ is a linear isomorphism from $\Lie{g}_\gamma \oplus \Lie{g}_\delta$ onto $\Lie{g}_\delta \oplus \Lie{g}_\gamma$ when $Z \neq 0$.

Either $\gamma = \delta$ or the roots $\gamma$ and $\delta$ span a root system of rank $2$.
By inspection of the possibilities, we see that the hypotheses that $\gamma$, $\delta$, and $\gamma+\delta$ are roots and $2\gamma+\delta$ and $2\gamma+\delta$ are not roots imply that $\| \gamma\| = \|\delta\|$ and $\lip \gamma+\delta,\delta \rip = \lip \gamma+\delta,\gamma \rip > 0$.
Now \eqref{[X,KX]} and \eqref{K2} follow immediately.
Further,  if $\gamma = \delta$ or neither $2\gamma$ nor $2\delta$ is a root, then $\lip \gamma+\delta,\gamma \rip = 1$, and so
\[
J_Z^2 X = -\|Z\|^2 X 
\quad\forall X \in \Lie{g}_\gamma \oplus \Lie{g}_\delta \quad\forall Z  \in \Lie{g}_{\gamma+\delta},
\]
as required.
\end{proof}

\begin{remark}\label{rem:2alpha}
We have just shown that the Iwasawa $\Lie{n}$ algebras of real rank one simple Lie algebras are $H$-type.
Further, inspection of the root systems of rank one and two shows that if $\gamma$, $\delta$, and $\gamma+\delta$ are roots and $2\gamma+\delta$ and $2\gamma+\delta$ are not roots, then either $2\gamma$ and $2(\gamma+\delta)$ are both roots, or neither is a root.
\end{remark}

\begin{cor}\label{cor:3.7}
Suppose that $\gamma$, $\delta$, and $\gamma+\delta$ are positive roots, and that neither $\gamma+2\delta$ nor $2\gamma+\delta$ is a root.
If $D$ is a root space preserving derivation of $\Lie{g}$ whose restriction to $\Lie{g}_{\gamma+\delta}$ is symmetric, then
\begin{equation}\label{eq:D*-localder}
D\transp [X,Y] = [D\transp X, Y] + [X, D\transp Y]
\qquad\forall X \in \Lie{g}_\gamma\quad\forall Y \in \Lie{g}_\delta.
\end{equation}
\end{cor}

\begin{proof}
The proof is a mild generalisation of the proof of Corollary \ref{cor:2.5a}.
Observe first that if $E$ is a root space preserving linear endomorphism of $\Lie{g}_{\gamma} \oplus \Lie{g}_{\delta} \oplus \Lie{g}_{\gamma+\delta}$, then $E[X,Y] = [EX,Y]+[X,EY]$ if and only if
\begin{equation}\label{eq:local-der}
\begin{aligned}{}
\lip J_{E\transp  Z} X, Y \rip 
&= \lip E\transp  Z, [ X, Y ] \rip
\\& = \lip Z, E [ X, Y ] \rip 
\\&= \lip Z, [ E X, Y ] \rip + \lip Z, [ X, E Y ] \rip 
\\&= \lip J_Z {E X}, Y \rip + \lip E\transp  J_Z X, Y \rip ,
\end{aligned}
\end{equation}
for all $X \in\Lie{g}_{\gamma}$, all $Y \in \Lie{g}_{\delta}$, and all $Z \in \Lie{g}_{\gamma+\delta}$.

Since $D\rist{\Lie{g}_{\alpha+\beta}}$ is symmetric, it is diagonalisable.
Take an eigenvector $Z$ in $\Lie{g}_{\alpha+\beta}$ with eigenvalue $2\mu$.
By \eqref{eq:local-der},
\[
2\mu J_Z = J_{DZ}  = J_{D\transp  Z} = D\transp  J_Z + J_Z D,
\]
whence composition on both sides by $J_Z$ gives
\[
-2\mu|Z|^2 J_Z =  -|Z|^2 J_ZD\transp  - |Z|^2 D J_Z,
\]
and
\[
J_{DZ} =  D J_Z +  J_ZD\transp .
\]
This holds for all eigenvectors $Z$ of $D$, and so for all $Z \in \Lie{z}$ by linearity,  so \eqref{eq:D*-localder} holds by \eqref{eq:local-der}. 
\end{proof}

We are supposing that $\Lie{g}$ is simple, so $\Sigma$ is indecomposable.
In particular, $\Sigma$ contains just one highest root (see Bourbaki \cite[p.~165, Proposition~25]{Bourbaki}), which we denote by $\omega$.
We fix the constant $c$ in \eqref{eq:3.1} by requiring that $\Vert \omega \Vert^2 = 2$. 
Then for each $\gamma \in \Sigma$, the number $\lip \gamma, \omega \rip$ is one of $\pm 2$, $\pm 1$ and $0$; further, it is $\pm 2$ if and only if $\gamma = \pm\omega$.

Define
\[
\Sigma_1 = \{ \gamma \in \Sigma \colon \lip \gamma, \omega \rip = 1 \}
\and
\Sigma_0 = \{ \gamma \in \Sigma \colon \lip \gamma, \omega \rip = 0 \} ,
\]
and write $\Sigma_{0}^+$ for $\Sigma^+ \cap \Sigma_{0}$. 
Then, by Ciatti \cite[Lemma~2.1]{C3},
\[
\Sigma^+ = \Sigma_{0}^+ \cup \Sigma_1 \cup \{ \omega \} .
\]
Further, define
\begin{equation*}
\Lie{v} = \sum_{\gamma \in \Sigma_1} \Lie{g}_{\gamma} ,
\quad
\Lie{h} = \Lie{v} \oplus \Lie{g}_{\omega}
\and
\Lie{n}_0 = \sum_{\gamma \in \Sigma_{0}^+} \Lie{g}_{\gamma} ;
\end{equation*}
then
\[
\Lie{n} = \Lie{n}_0 \oplus \Lie{v} \oplus \Lie{g}_{\omega} = \Lie{n}_0 \oplus \Lie{h} .
\]

Following Ciatti \cite{C3}, for $Z$ in $\Lie{g}_{\omega}$, we define the operator $J_Z \colon \Lie{v} \rightarrow \Lie{v}$ by
\begin{equation}\label{eq:3.10}
J_Z X = [ Z, \theta X ] .
\end{equation}
Then by definition and \eqref{eq:3.2},
\[
\lip J_Z X, Y \rip  = \lip [Z, \theta X],Y \rip = \lip Z, [X,Y] \rip 
\qquad\forall X, Y \in \Lie{v}.
\]

\begin{lemma}[{Ciatti \cite{C1}}]\label{ciatti-htype}
The pair $(\Lie{v}\oplus\Lie{z} , \lip \cdot, \cdot \rip)$ is an $H$-type algebra with centre $\Lie{g}_{\omega}$, that is, $[\Lie{v}, \Lie{v}] = \Lie{g}_\omega$ and
\begin{equation}\label{eq:3.11}
J_Z^2 X = - \Vert Z \Vert^2 X
\end{equation}
for all $Z$ in $\Lie{g}_{\omega}$ and $X$ in $\Lie{v}$.
\end{lemma}

\begin{proof}
This follows from Lemma \ref{lemH}.
\end{proof}

Now we list some $H$-type subalgebras of the Iwasawa $\Lie{n}$ algebras of rank two simple Lie algebras.
When the root system is of type $A_2$, then $\Lie{n}$ is itself an $H$-type algebra.  
When the root system is of type $B_2$, say $\Sigma = \{ \alpha, \beta, \alpha+\beta, 2\alpha+\beta\}$, then $\Lie{g}_{\alpha} \oplus \Lie{g}_{\alpha+\beta} \oplus \Lie{g}_{2\alpha+\beta}$ is an $H$-type subalgebra (and $\Lie{g}_{\beta} \oplus \Lie{g}_{\alpha+\beta} \oplus \Lie{g}_{2\alpha+\beta}$ is abelian and hence a degenerate $H$-type algebra too). 
When the root system is of type $BC_2$, say $\Sigma = \{ \alpha, 2\alpha, \beta, \alpha+\beta, 2\alpha+\beta, 2\alpha+2\beta\}$, then $\Lie{g}_{\alpha} \oplus \Lie{g}_{2\alpha}$ and $\Lie{g}_{\beta} \oplus \Lie{g}_{\alpha+\beta} \oplus \Lie{g}_{2\alpha+\beta} \oplus \Lie{g}_{2\alpha+2\beta}$ are $H$-type subalgebras, and $\Lie{g}_{\alpha} \oplus \Lie{g}_{\alpha+\beta} \oplus \Lie{g}_{2\alpha+\beta}$ is close to an $H$-type subalgebra (see Lemma \ref{lemH}).
Finally, when the root system is of type $G_2$, say $\Sigma = \{ \alpha, \beta, \alpha+\beta, 2\alpha+\beta, 3\alpha+\beta,3\alpha+\beta\}$, then $\Lie{g}_{\alpha} \oplus \Lie{g}_{2\alpha+\beta}\oplus \Lie{g}_{3\alpha+\beta}$ and $\Lie{g}_{\beta} \oplus \Lie{g}_{\alpha+\beta} \oplus \Lie{g}_{2\alpha+\beta} \oplus \Lie{g}_{3\alpha+\beta}\oplus \Lie{g}_{3\alpha+2\beta}$ are $H$-type subalgebras.

\subsection{The fine structure of $\Lie{g}$.}
We now study $\Lie{g}$ in more detail.

\begin{lemma}\label{lem:UX}
Suppose that $\gamma$,  $\delta$, and $\gamma+\delta$ are positive roots, and that $\gamma-\delta$ and $\gamma+2\delta$ are not roots.
If $U \in \Lie{g}_\delta \setminus\{0\}$, then
\begin{equation}\label{eq:surj}
\{ [U,X]  : X \in \Lie{g}_{\gamma} \} = \Lie{g}_{\gamma+\delta}
\and
\{ [\theta U,Y]  : Y \in \Lie{g}_{\gamma+\delta} \} = \Lie{g}_{\gamma}.
\end{equation}
Consequently, $\dim(\Lie{g}_{\gamma}) = \dim(\Lie{g}_{\gamma+\delta})$ and $\ad(U)$ is bijective from $\Lie{g}_{\gamma}$ to $\Lie{g}_{\gamma+\delta}$.
\end{lemma}

\begin{proof}
The hypotheses imply that $\lip \gamma+\delta , \delta \rip = \frac{1}{2} \lip\delta, \delta \rip \neq 0$.
Evidently, if $Y \in \Lie{g}_{\gamma+\delta}$, then $[\theta U, Y ] \in \Lie{g}_{\gamma}$ and
 \[
 [U , [\theta U, Y]] =  [[U , \theta U], Y] +   [\theta U, [U ,Y]] = (\gamma+\delta)([U , \theta U]) Y = - \lip \gamma+\delta, \delta \rip \| U \|^2 Y
 \]
 by \eqref{eq:3.3}, and it follows that $Y$ is in the range of $\ad(U)$.  
 This proves the left hand formula of \eqref{eq:surj} and hence $\dim(\Lie{g}_{\gamma}) \geq \dim(\Lie{g}_{\gamma+\delta})$.
 The right hand formula and the opposite inequality $\dim(\Lie{g}_{\gamma}) \leq \dim(\Lie{g}_{\gamma+\delta})$ may be shown similarly.
 
The bijectivity of $\ad(U)$, and of $\ad(\theta U)$, follow.
\end{proof}

\begin{lemma}\label{lem:UXX}
Suppose that $\gamma$,  $\delta$, $\gamma+\delta$ and $\gamma+2\delta$ are positive roots, and that $\gamma-\delta$ and $\gamma+3\delta$ are not roots.
If $U \in \Lie{g}_\gamma \setminus\{0\}$ and $X \in \Lie{g}_\delta\setminus\{0\}$, then
\begin{equation*}\label{eqB2}
[U,X] \neq 0
\and
[[U,X],X] \neq 0.
\end{equation*}
\end{lemma}

\begin{proof}
First, $[U, \theta X] = 0$ since $\gamma - \delta$ is not a root.

The hypotheses imply that $\gamma$ and $\delta$ span a root subsystem of type $B_2$ or $BC_2$, whence $2\gamma+\delta$ is not a root, and that $\lip\gamma + \delta, \delta \rip = 0$ while $\lip \gamma, \delta \rip \neq 0$ (see Bourbaki \cite[p.~148, Th\'eor\`eme 1]{Bourbaki}). 
Now $[U,X] \neq 0$, by Lemma \ref{lem:UX} with the roles of $\gamma$ and $\delta$ exchanged.

Next, by the Jacobi identity and the facts that $\lip \gamma + \delta , \delta\rip =0$ and $[U, \theta X] = 0$, 
\begin{align*}
[[[U,X],X],\theta X]
&=
[[\theta X,X],[U,X]] + [[[U,X],\theta X],X]
\\
&=
(\gamma + \delta )([\theta X,X])[U,X]+ [[[U,X],\theta X],X]
\\
&=
\lip \gamma + \delta , \delta \rip \|X\|^2 [U,X] + [[U,[X, \theta X]],X] + [[[U,\theta X],X],X]
\\
&= \lip \gamma, \delta \rip \|X\|^2 [U,X]\neq 0,
\end{align*}
which ensures that $[[U,X],X] \neq 0$ if neither $U$ nor $X$ is $0$.
\end{proof}

\begin{lemma}\label{lem:root-space-commutators}
Suppose that $\gamma$, $\delta$ and $\gamma+\delta$ are positive roots.
Then
\begin{gather*}
\{ [X, Y] : X \in \Lie{g}_\gamma, \ Y \in \Lie{g}_\delta \} = \Lie{g}_{\gamma+\delta} ; \\
\{ U \in \Lie{g}_\gamma : \ad(U)\rist{\Lie{g}_\delta} = 0 \} = \{0\} .
\end{gather*}
\end{lemma}

\begin{proof}
Observe that $\ad(\Lie{m} \oplus \Lie{a})$ is irreducible on $\Lie{g}_{\gamma+\delta}$ (for instance, this follows from Kostant's double transitivity theorem \cite{Kos}; see also Cowling, Dooley, Kor\'anyi and Ricci \cite{CDKR1}).
Thus the subspace $[\Lie{m} \oplus \Lie{a}, [\Lie{g}_\gamma, \Lie{g}_\delta] ]$ is either $\Lie{g}_{\gamma+\delta}$ or $\{0\}$. 
Hence, to prove the first equality, it suffices to show that 
\begin{equation}\label{eq:commutator-not-null}
[\Lie{g}_\gamma, \Lie{g}_\delta] \neq\{0\}.
\end{equation}
To do this, we consider the subset $(\Z \gamma + \Z \delta) \cap \Sigma$ of $\Sigma$, which is a root system in its own right.

If this root subsystem is of rank one, then necessarily $\delta=\gamma$, and $\gamma + \delta = 2\gamma$.
In this case, $\Lie{g}_\gamma \oplus  \Lie{g}_{2\gamma} $ is an $H$-type algebra, and \eqref{eq:commutator-not-null} follows.

If the root system is of type $A_2$, then we are done, since $\Lie{g}_\gamma \oplus \Lie{g}_\delta \oplus \Lie{g}_{\gamma+\delta} $ is an $H$-type algebra.

If the root system is of type $B_2$ or $BC_2$, then \eqref{eq:commutator-not-null} follows from Lemma \ref{lem:UXX}.

If the root system is of type $G_2$, then the algebra is split or complex, and in this case the result is well known.

Finally, suppose that $U \in \Lie{g}_\gamma\setminus \{0\}$ and $[U,X ] = 0$ for all $X \in \Lie{g}_\delta$.
Then 
\[
[ \ad(W) U, X] = \ad(W) [U,X] - [U, \ad(W) X] = 0
\]
for all $X \in \Lie{g}_\delta$ and all $W \in \Lie{m} \oplus \Lie{a}$, and hence $[V,X ] = 0$ for all $V \in \Lie{g}_{\gamma}$ and all $X \in \Lie{g}_\delta$, which is impossible.
\end{proof}


We are going to analyse general simple Lie algebras by looking carefully at subalgebras of rank $1$ or $2$. 
Given a subset $\Epsilon$ of $\Sigma^+$, we write $\Lie{g}^{\Epsilon}$ for the subalgebra of $\Lie{g}$ generated by the root spaces $\Lie{g}_{\epsilon}$ where $\epsilon$ ranges over $\rspan \Epsilon$.

We define, for any root $\gamma$, $\Lie{m}^{\{\gamma\}} = \Lie{m} \cap \Lie{g}^{\{\gamma\}}$ and 
\[
\Lie{m}^{\gamma} 
= \rspan \{ [ X, \theta Y ] \colon X, Y \in \Lie{g}_{\gamma}, \lip X, Y \rip = 0 \} .
\]

\begin{lemma}\label{lem:ideals-in-m}
The following hold:
\begin{enumerate}[{\rm(i)}]
\item
if $\gamma$ is a root, then $\Lie{m}^{-\gamma} = \Lie{m}^{\gamma}$,
\item 
if $\gamma$ is a root, then $[\Lie{m} , \Lie{m}^{\gamma}] \subseteq \Lie{m}^{\gamma} $, 
\item
if $\gamma$, $\delta$ and $\epsilon$ are roots and $\epsilon \in  \Z\gamma + \Z\delta$, then $\Lie{m}^{\epsilon} \subseteq \Lie{m}^\gamma + \Lie{m}^\delta$, 
\item
if $\gamma$ is a root, then $\Lie{m}^{\gamma} \subseteq \Lie{m}^{\{\gamma\}}$, with equality if $\gamma/2$ is not a root,
\item
$\Lie{m} = \sum_{\gamma \in \Delta} \Lie{m}^\gamma = \sum_{\gamma \in \Sigma^+} \Lie{m}^\gamma$. 
\end{enumerate}
\end{lemma}

\begin{proof}
Observe that if $X, Y \in \Lie{g}_{-\gamma}$, then $[X, \theta Y] = - [ \theta Y, \theta(\theta X)]$, and  $\theta Y, \theta X \in \Lie{g}_{\gamma}$, so (i) holds.

Now we prove (ii).
If $Z \in \Lie{m}$ and $X, Y \in \Lie{g}_\gamma$, then
\[
\begin{aligned}{}
 [Z, [X, \theta Y]] = [ [Z, X], \theta Y]  +  [X, [Z, \theta Y]] . 
\end{aligned}
\]
Both summands lie in $\Lie{m}^{\gamma}$. 
Thus $\Lie{m}^{\gamma}$ is an ideal in $\Lie{m}$, and in particular, is a subalgebra.

Next, we prove (iii).
First, if $\gamma$, $\delta$ and $\gamma+\delta$ are roots and $W, Z \in \Lie{g}_{\gamma+\delta}$, then there exist $X \in \Lie{g}_\gamma$ and $Y \in \Lie{g}_\delta$ such that $[X,Y] = Z$, by Lemma \ref{lem:root-space-commutators}.
Then
\[
\begin{aligned}{}
[W, \theta Z] = [W, [\theta X,\theta Y]] 
&= [[W, \theta X],\theta Y] +  [\theta X,[W,\theta Y]] \\
&= [W, \theta X],\theta Y] -  [[W,\theta Y], \theta X] \in \Lie{m}^{\gamma} + \Lie{m}^{\delta}.
\end{aligned}
\]

To prove (iii), we use (i) and the observation above repeatedly.

To prove (iv), observe that if $2\gamma$ and $\frac{1}{2}\gamma$ are not roots, then $\Lie{g}_{-\gamma} \oplus \Lie{m}^{\gamma} \oplus \R H_\gamma \oplus \Lie{g}_{\gamma}$ coincides with the subalgebra  $\Lie{g}^{\{\gamma\}}$, whence $\Lie{m}^{\gamma} = \Lie{m}^{\{\gamma\}}$.
Similarly, if $\gamma$ and $2\gamma$ are both roots, then $\Lie{m}^{2\gamma} \subseteq \Lie{m}^\gamma$, by (ii), so $\Lie{g}_{-2\gamma} \oplus \Lie{g}_{-\gamma}  \oplus \Lie{m}^{\gamma} \oplus \R H_\gamma \oplus \Lie{g}_{\gamma} \oplus \Lie{g}_{2\gamma}$ coincides with the subalgebra $\Lie{g}^{\{\gamma\}}$, and again $\Lie{m}^{\gamma} = \Lie{m}^{\{\gamma\}}$. 
Finally, if $\gamma$ and $\frac{1}{2}\gamma$ are both roots, $\Lie{m}^{\gamma} \subseteq \Lie{m}^{\gamma/2}$, and $\Lie{m}^{\gamma} \subseteq \Lie{m}^{\{\gamma\}}$.
This inclusion is strict when $\Lie{g}^{\{\gamma\}}$ is $\Lie{su}(n,1)$ (where $n >1$) or $\Lie{sp}(n,1)$ (where $n>1$).

To prove (v), we use (i) and (iii) repeatedly.
\end{proof}

\section{Derivations of semisimple Lie algebras.}
In this section, we discuss the height of roots and the associated grading of the Lie algebra $\Lie{g}$, and prove a number of results on height preserving derivations.
Then we prove a localisation result for derivations of $\Lie{g}$.
Our final result is a necessary and sufficient condition for a skew-symmetric root space preserving derivation of $\Lie{n}$ to be of the form $\ad(Z)$ for some $Z \in \Lie{m}$.

\begin{lemma}\label{lem:W=0}
Suppose that $W \in \Lie{g}_0$.
Then $W = 0 $ if and only if $\ad(W) \rist{\Lie{g}_\beta} = 0$ for all $\beta \in \Delta$.
\end{lemma}

\begin{proof} 
One implication is obvious.
To prove the other, suppose that $\ad(W) \rist{\Lie{g}_\beta} = 0$ for all $\beta \in \Delta$.
Then $\ad(W)$ vanishes on $\Lie{n}$, whence $\ad(\theta W)$ also vanishes on $\Lie{n}$ since $\ad(\theta W) = - \ad(W)\transp$.

Now if $X \in \Lie{n}$, then
\[ 
[ W , \theta X ] = \theta [ \theta W , X ] = 0.
\]
Since $\Lie{g}$ is simple, and $\ad(W)$ is a derivation that vanishes on $\Lie{n} \oplus\theta\Lie{n}$ and hence on the algebra that this generates, that is, $\Lie{g}$, we conclude that $W = 0$.
\end{proof}

We are interested in the derivations $D$ of $\Lie{n}$ that preserve the root space structure, that is, are such that $D(\Lie{g}_\alpha) \subseteq \Lie{g}_\alpha$ for all $\alpha \in \Sigma^+$.
We write $\Der{\Lie{n}}$ for the space of these mappings.

Recall that the height of the positive root $\alpha$, written $\height(\alpha)$, is defined to be $\sum_{j=1}^{r} n_j$, where $\alpha = \sum_{j=1}^r n_j \alpha_j$ and $\alpha_j \in \Delta$. 
Note that there is an element $H_0$ of $\Lie{a}$ such that $[H_0,X] = \height(\alpha) X$ for all $X \in \Lie{g}_\alpha$ and all $\alpha \in \Delta$.
We may extend the height function to all roots: we set $\height(\gamma) = h$ when $[H_0,X] = hX$ for all $X \in \Lie{g}_\gamma$.
When $h$ is a nonzero integer, we write $\Lie{g}_{h}$ for $\sum_{\gamma} \Lie{g}_\gamma$, where we sum over the $\gamma\in \Sigma$ such that $\height(\gamma) = h$ otherwise.
We defined $\Lie{g}_0$ to be the ``null root space'' $\Lie{m} \oplus \Lie{a}$, which fortunately coincides with the subspace of $\Lie{g}$ of elements of height $0$, and so $\Lie{g}_0$ may also be used to describe the latter space, consistently with our $\Lie{g}_h$ notation.

\begin{prop}\label{prop:strat}
The Lie algebra $\Lie{g}$ is graded: more precisely, $\Lie{g} = \sum_{h\in \Z} \Lie{g}_h$, and $[\Lie{g}_h, \Lie{g}_{h'}] \subseteq \Lie{g}_{h+h'}$. 
 Next, $\Lie{n}$ is stratified, that is, $[\Lie{g}_h, \Lie{g}_{1}] = \Lie{g}_{h+1}$ for all $h \in \Z^+$, so $\Lie{g}_1$ generates $\Lie{n}$.
 Finally, if $0 < h < \height(\omega)$, then $\{ X \in \Lie{g}_h : \ad(X) \rist{\Lie{g}_{1}} = 0\} = \{0\}$.
 \end{prop}

\begin{proof}
The linear operator $\ad(H_0)$ on $\Lie{g}$ is diagonalisable, whence $\Lie{g}$ decomposes as a sum of eigenspaces; given that the simple roots correspond to eigenvalue $1$ and all positive roots are sums of simple roots (with multiplicities), all eigenvalues are integers.
Further, $\ad(H_0)$ is a derivation and so $[\Lie{g}_h, \Lie{g}_{h'}] \subseteq \Lie{g}_{h+h'}$.

If $\height(\gamma) = h+1$ where $h>0$, then there exists $\alpha \in \Delta$ such that $\gamma - \alpha$ is a root, by \cite[Lemma 3.1]{CDMKR}.
Lemma \ref{lem:root-space-commutators} shows that $[\Lie{g}_{\alpha}, \Lie{g}_{\gamma - \alpha} ] = \Lie{g}_\gamma$, and it follows that $\Lie{g}_\gamma \subseteq [\Lie{g}_1, \Lie{g}_{h}] $.
This applies to all $\gamma$ of height $h+1$ and so $\Lie{g}_{h+1} \subseteq [\Lie{g}_1, \Lie{g}_{h}] $.
The converse inclusion has already been established.

Finally, suppose that $X \in \Lie{g}_h$ and $\ad(X)\rist{\Lie{g}_1} =0$.
Write $X$ as $\sum_{\gamma} X_\gamma$, where $X_\gamma \in \Lie{g}_\gamma$ and $\height(\gamma) = h$.
Now 
\[
0 = [[H,X],Y] = [[H,Y],X] + [H,[X,Y]] 
\qquad\forall Y \in \Lie{g}_1,
\]
whence $\ad([H,X])\rist{\Lie{g}_1} =0$ for all $H \in \Lie{a}$.
The algebra of operators generated by the operators $\ad(H)$ for all $H \in \Lie{a}$ is closed under transpose and hence spanned by its minimal projections, which are precisely the projections onto the root spaces $\Lie{g}_\gamma$ as $\gamma$ varies over $\Sigma$.
We deduce that $\ad(X_\gamma)\rist{\Lie{g}_1} =0$ for all $\gamma$ of height $h$.
By Lemma \ref{lem:root-space-commutators}, each $X_\gamma$ is zero.
\end{proof}

We are now going to work with derivation identities.
\begin{definition}\label{def:DandE}
For $\gamma , \delta \in \Sigma^+$, let $(D_{\gamma, \delta})$ be the formula
 \[
 D[X,\theta Z] = [DX,\theta Z] + [X,\theta DZ]
 \]
 for all $X \in \Lie{g}_\gamma$ and all $Z \in \Lie{g}_\delta$,  and $(E_{\gamma, \delta})$ be the formula
  \[
 D[[X,\theta Y],Z] = [[DX,\theta Y],Z] + [[X,\theta DY],Z] + [[X,\theta Y],DZ]
 \]
 for all $X, Y \in \Lie{g}_\gamma$ and all $Z \in \Lie{g}_\delta$.
 \end{definition}
 
 Note that if $D$ were a derivation of $\Lie{g}$ such that $\theta D = D\theta$, then these formulae would follow from the Jacobi identity.
 At this point, we are not asserting that these are true!

\begin{theorem}\label{thm:construct-der}
 Suppose that $D$ is a skew-symmetric height preserving derivation of $\Lie{n}$, and that $(E_{\gamma, \delta})$, as in Definition \ref{def:DandE}, holds for all $\gamma,\delta\in\Delta$.
 Then the following statements hold. 
 \begin{enumerate}[(i)]
 \item
 There is a unique well-defined linear map $\tilde D: \Lie{g}_0 \to \Lie{g}_0$ such that 
 \[
 \tilde D[X,\theta Y] = [DX, \theta Y] + [X, \theta DY]
 \qquad\forall X, Y \in \Lie{g}_1.
 \]
 \item
 The range of the linear map $\tilde D$ is contained in $\Lie{m}$.
\item
The linear map $E : \Lie{g}_0 \oplus \Lie{n} \to \Lie{g}_0 \oplus \Lie{n}$, defined by 
 \[
 E(W + X) = \tilde DW + DX
 \qquad\forall W \in \Lie{g}_0\quad\forall X \in \Lie{n},
 \]
 is a derivation. 
\item 
If $h \geq  k \geq 0$, then 
 \[
 E[U,\theta V] = [EU, \theta V] + [U, \theta E V] 
 \qquad\forall U \in \Lie{g}_h \quad\forall V \in \Lie{g}_k
 \] 
 \end{enumerate}
 \end{theorem}

\begin{proof}
 To prove (i), we first claim that if $\alpha, \beta \in \Delta$ and $\alpha \neq \beta$, then 
\begin{equation}\label{eq:DXthetaY+XthetaDY}
  [DX, \theta Y] + [X, \theta DY] = 0
 \qquad\forall X \in\Lie{g}_\alpha \quad\forall Y \in \Lie{g}_\beta.
\end{equation}
 To see this, take $W$ in $\Lie{g}_0$ of the form $[U, \theta V]$, where $U, V \in \Lie{g}_\gamma$, for some $\gamma \in \Delta$.
 Since $D$ is a skew-symmetric derivation, 
 \[
\begin{aligned}
 \lip [DX, \theta Y] + [X, \theta DY], W \rip
& = - \lip DX, [W, Y] \rip - \lip X, [W, DY] \rip \\
& =  \lip X, D[W, Y] \rip - \lip X, [W, DY] \rip \\
& =  \lip X, [[DU,\theta V]+ [U, \theta DV], Y] \rip \\
& =  - \lip [X, \theta Y],  [DU,\theta V]+ [U, \theta DV]\rip \\
&=0,
\end{aligned} 
\] 
since $[X, \theta Y] = 0$ because $\alpha - \beta$ is not a root; the third step uses $(E_{\gamma, \beta})$.
Since $\Lie{g}_0$ is spanned by elements of the form $[U, \theta V]$, our claim is established.

Now we define $L: \bigcup_{\alpha \in \Delta} \Lie{g}_\alpha \times \bigcup_{\alpha \in \Delta} \Lie{g}_\alpha  \to \Lie{g}_0$ by $L(X,Y) = [X, \theta Y]$.
Then $L$ extends automatically to a linear map, also denoted $L$, from $ \Lie{g}_1 \otimes \Lie{g}_1$ to  $\Lie{g}_0$.
Take $X_j, Y_j \in \Lie{g}_{1}$, and suppose that $\sum_j [X_j, \theta Y_j] = 0$ in $\Lie{g}_0$.
Write each $X_j$ as $\sum_{\alpha} X_{j,\alpha}$ and each $Y_j$ as $\sum_{\beta} Y_{j,\beta}$, where $X_{j,\alpha} \in \Lie{g}_\alpha$ and $Y_{j,\beta} \in\Lie{g}_\beta$; here $\alpha$ and $\beta$ range over $\Delta$.
Then
\[
\sum_j [X_j, \theta Y_j] = \sum_{j,\alpha} [X_{j,\alpha}, \theta Y_{j,\alpha}] 
\] 
since $[X_{j,\alpha}, \theta Y_{j,\beta}] = 0$ because $\alpha - \beta$ is not a root if $\alpha \neq \beta$.
If $\gamma \in \Delta$ and $W \in \Lie{g}_\gamma$, then
 \[
\sum_j [[X_j, \theta Y_j] , W] = 0
\and
\sum_j [[X_j, \theta Y_j] , DW] = 0
 \]
 by hypothesis.
Thus by $(E_{\alpha,\gamma})$, 
 \[
\begin{aligned}
 0
 &= D \sum_j [[X_j, \theta Y_j] , W] 
 = D\sum_{j,\alpha} [[X_{j,\alpha}, \theta Y_{j,\alpha}] , W] \\
 &= \sum_{j,\alpha} [[DX_{j,\alpha}, \theta Y_{j,\alpha}] , W]  + [[X_{j,\alpha}, \theta DY_{j,\alpha}] , W]  + [[X_{j,\alpha}, \theta Y_{j,\alpha}] , DW] \\
 &= \sum_{j,\alpha} \bigl[ [DX_{j,\alpha}, \theta Y_{j,\alpha}]   + [X_{j,\alpha}, \theta DY_{j,\alpha}] , W\bigr]  + \sum_{j} [[X_{j}, \theta Y_{j}] , DW] \\
 &= \sum_{j,\alpha,\beta} \bigl[ [DX_{j,\alpha}, \theta Y_{j,\beta}]   + [X_{j,\alpha}, \theta DY_{j,\beta}] , W\bigr]   ,
\end{aligned}
\]
and this shows that
\[
\sum_{j} \bigl[ [DX_{j,\alpha}, \theta Y_{j,\beta}]   + [X_{j,\alpha}, \theta DY_{j,\beta}] , W\bigr] = 0.
\]
From Lemma \ref{lem:W=0}, we see that 
\[
\sum_{j} [DX_{j,\alpha}, \theta Y_{j,\beta}]   + [X_{j,\alpha}, \theta DY_{j,\beta}] =0.
\]
It follows immediately that $\tilde D$, given by
\[
 \tilde D\sum_j [X_j,\theta Y_j] = \sum_j \bigl( [DX_j, \theta Y_j] + [X_j, \theta DY_j] \bigr),
\]
is well-defined; clearly $\tilde D$ is also unique.

To prove (ii), note that if $X, Y \in \Lie{g}_\alpha$ where $\alpha \in \Delta$, and $H \in \Lie{a}$, then 
\[
\begin{aligned}
\lip [DX, \theta Y] + [X, \theta DY] , H \rip
&= \lip DX, [H,Y] \rip + \lip X,  [H, DY] \rip  \\
&= \lip DX, [H,Y] \rip + \lip [H,X],  DY \rip \\
&= \alpha(H) \lip DX, Y \rip + \lip X,DY \rip  \\ 
&= 0 ,
\end{aligned}
\]
since $D$ is skew-symmetric.
This equality now holds for all $X, Y \in \Lie{g}_1$ by linearity and \eqref{eq:DXthetaY+XthetaDY}, and so the range of $\tilde D$ is contained in $\Lie{m}$.

We now extend $D$ and $\tilde D$ to a linear map $E$ on $\Lie{g}_0 + \Lie{n}$ by setting $E(W+X) = \tilde DW + DX$ for all $W \in \Lie{g}_0$ and all $X \in \Lie{n}$.
Since $D$ is a derivation on $\Lie{n}$, to show that $E$ is a derivation it suffices to show that
\begin{equation}\label{eq:E-der-on-n}
D[W,X] = [\tilde D W,X] +  [W,  DX]
\qquad\forall W \in \Lie{g}_0 \quad\forall X \in \Lie{n}.
\end{equation}
and
\begin{equation}\label{eq:E-der-on-m}
\tilde D [W,U] = [\tilde D W,U] + [W, \tilde DU] 
\quad\forall W, U \in \Lie{g}_0. 
\end{equation}
To prove \eqref{eq:E-der-on-n}, observe that 
\[
D[W,X] - [\tilde D W,X] -  [W,  DX] = \bigl( [D, \ad(W)] - \ad(\tilde DW) \bigr) X, 
\]
and $[D, \ad(W)] - \ad(\tilde DW) $ is a derivation. 
To show that it is $0$ on $\Lie{n}$, it suffices to show that it vanishes on $\Lie{g}_\beta$ for all simple roots $\beta$.
By linearity, it suffices to take $W$ of the form $[X, \theta Y]$ where $X, Y \in \Lie{g}_\alpha$ and $\alpha \in \Delta$; this case follows from $(E_{\alpha,\beta})$.

To prove \eqref{eq:E-der-on-m}, we may suppose by linearity that $U = [X, \theta Y] $ where $X, Y \in \Lie{g}_\alpha$ for some $\alpha \in \Delta$.
Now $\theta \tilde D \theta = \tilde D$, so
\[
\begin{aligned}
\tilde D [W, [X, \theta Y]] 
&= \tilde D[[W,X], \theta Y] + \tilde D[X, [W, \theta Y]]  \\
&= \tilde D[[W,X], \theta Y] + \tilde D[X, \theta [\theta W,Y]]  \\
&= [D[W,X], \theta Y] + [[W,X], \theta DY] \\
&\quad + [DX, [W, \theta Y]] + [X, \theta D[\theta W,Y]]  \\
&= [[\tilde DW,X], \theta Y] + [[W,DX], \theta Y] + [[W,X], \theta DY] \\
&\quad + [DX, [W, \theta Y]] + [X, \theta [\tilde D\theta W,Y]] + [X, \theta [\theta W,DY]]  \\
&= [[\tilde DW,X], \theta Y] + [[W,DX], \theta Y] + [[W,X], \theta DY] \\
&\quad + [DX, [W, \theta Y]] + [X, [\tilde D W, \theta Y]] + [X,  [ W, \theta DY]]  \\
&= [\tilde DW,[X, \theta Y]] + [W,[DX, \theta Y]] + [W,[X, \theta DY]] \\
&= [\tilde DW,[X, \theta Y]] + [W,\tilde D [X, \theta Y]] ,
\end{aligned}
\]
and \eqref{eq:E-der-on-m} holds.

Finally, we prove (iv), using induction on $h$ and $k$.
We need to prove the identity $(D_{h,k})$, given by
\[
E [X, \theta Y ] = [EX, \theta Y] + [X, \theta EY]
\qquad\forall X \in \Lie{g}_{h} \quad\forall Y \in \Lie{g}_{k}.
\]

First we suppose that $k=1$.
The identity $(D_{h,1})$ is equivalent to
\[
 [E [X, \theta Y], Z] = [[EX, \theta Y], Z] + [[X, \theta EY],Z] 
\]
for all $X \in \Lie{g}_{h}$, all $Y\in \Lie{g}_{1}$ and all $Z\in \Lie{g}_{1}$, by Proposition \ref{prop:strat}.
Write $W$ for  $[X,Z]$ in $\Lie{g}_{h+1}$.
Since $E$ is a derivation, from the Jacobi identity and the definition of $E$
\[
\begin{aligned}{}
&   [E[X, \theta Y],Z] - [[EX, \theta Y], Z] - [[X, \theta EY],Z] \\
&\quad= E[[X, \theta Y],Z] - [[X, \theta Y], EZ] - [[EX, \theta Y], Z] - [[X, \theta EY],Z] \\
&\quad= E[[X, Z], \theta Y] + E[X, [\theta Y,Z]] - [[X, \theta Y], EZ] - [[EX, \theta Y], Z] - [[X, \theta EY],Z] \\
&\quad= E[[X, Z], \theta Y] + [EX, [\theta Y,Z]] + [X, [\theta EY,Z]] + [X, [\theta Y,EZ]]  \\
&\quad\qquad- [[X, \theta Y], EZ] - [[EX, \theta Y], Z] - [[X, \theta EY],Z] \\
&\quad= E[[X, Z], \theta Y] + [\theta Y, [EX, Z]] +  [\theta EY, [X,Z]] + [\theta Y, [X, EZ]]  \\
&\quad= E[W, \theta Y]  - [EW,\theta Y]  - [W, \theta EY] .
\end{aligned}
\]
We deduce that if $(D_{h+1,1})$ holds, the last line vanishes, hence the first line vanishes, and $(D_{h,1})$ holds.
Since $(D_{h,1})$ holds for large positive $h$ (because there is nothing to prove as $\Lie{g}_h = \{0\}$), $(D_{h,1})$ holds for all positive $h$.

Now suppose that $(D_{h,1})$ and $(D_{h,k})$ hold where $1 \leq k < h$.
Take $X \in \Lie{g}_h$, $Y_1 \in \Lie{g}_{1}$ and $Y_2 \in \Lie{g}_{k}$. 
Then
\[
\begin{aligned}
&E [X, \theta [Y_1,Y_2] ] - [EX, \theta [Y_1,Y_2]] - [X, \theta E [Y_1,Y_2]] \\
&\quad= E [[X, \theta Y_1],\theta Y_2] + E [ \theta Y_1,[X,\theta Y_2]] - [[EX, \theta Y_1],\theta Y_2] -  [\theta Y_1, [EX,\theta Y_2]] \\
&\qquad  - [X,  [\theta EY_1,\theta Y_2]] - [X, [\theta Y_1, \theta EY_2]] \\
&\quad= [E [X, \theta Y_1],\theta Y_2] + [[X, \theta Y_1],\theta EY_2] +  [ \theta EY_1,[X,\theta Y_2]] +  [ \theta Y_1,E [X,\theta EY_2]] \\
&\qquad   - [[EX, \theta Y_1],\theta Y_2]  -  [\theta Y_1, [EX,\theta Y_2]]  -  [X,  [\theta EY_1,\theta Y_2]] - [X, [\theta Y_1, \theta EY_2]]  \\
&\quad= [ [EX, \theta Y_1],\theta Y_2] + [[X, E \theta Y_1],\theta Y_2] + [[X, \theta Y_1], \theta EY_2] +  [ \theta EY_1,[X,\theta Y_2]] \\
&\qquad   +  [ \theta Y_1, [EX,\theta Y_2]]  +  [ \theta Y_1, [X,\theta EY_2]]  - [[EX, \theta Y_1],\theta Y_2]  -  [\theta Y_1, [EX,\theta Y_2]] \\
&\qquad   -  [X,  [\theta EY_1,\theta Y_2]]  - [X, [\theta Y_1,E \theta Y_2]]\\
&\quad=  [[X, E \theta Y_1],\theta Y_2] + [[X, \theta Y_1],\theta EY_2] +  [ \theta EY_1,[X,\theta Y_2]] +  [ \theta Y_1, [X,\theta EY_2]] \\
&\qquad   -  [X,  [\theta EY_1,\theta Y_2]] - [X, [\theta Y_1,E Y_2]]\\
&\quad = 0 .
\end{aligned}
\]
By Proposition \ref{prop:strat}, $(D_{h,k+1})$ also holds.
By induction, $(D_{h,k})$ holds whenever $h \geq k \geq 0$.
 \end{proof}

\begin{theorem}\label{thm:skew-der-char}
 Suppose that $D$ is a skew-symmetric height preserving derivation of $\Lie{n}$.
  Then the following are equivalent:
 \begin{enumerate}[(i)]
 \item
 there exists a height preserving derivation $\tilde D$ of $\Lie{g}$ whose restriction to $\Lie{n}$ coincides with $D$;
 \item
 $D = \ad(W)$ for some $W \in \Lie{m}$;
 \item
$(E_{\gamma, \delta})$ holds for all $\gamma,\delta\in \Sigma^+$.
 \item
$(E_{\gamma, \delta})$ holds for all $\gamma,\delta\in\Delta$.
 \end{enumerate}
 \item
 Further,  if any of these conditions hold, then $\tilde D$ is root space preserving.
\end{theorem}

\begin{proof}
Suppose that (i) holds.
Since all derivations of $\Lie{g}$ are inner, $\tilde D = \ad(W)$ for some $W \in \Lie{g}$.
Evidently $\ad(W)$ preserves height if and only if $W \in \Lie{g}_0$. 
Thus $W  \in \Lie{m} \oplus \Lie{a}$.
Since $D$ is skew-symmetric, $W \in \Lie{m}$, and (ii) is proved.

If (ii) holds, then the Jacobi identity and the fact that $\theta W = W$ imply that
\begin{align*}
&D[[X,\theta Y],Z] \\
&\quad= \ad(W) [[X,\theta Y],Z] \\
&\quad=  [[\ad(W)X,\theta Y],Z] + [[X \ad(W)\theta Y],Z] + [[X,\theta Y],\ad(W)Z] \\
&\quad= [[\ad(W)X,\theta Y],Z] + [[X, \theta \ad(W)Y],Z] + [[X,\theta Y],\ad(W)Z]  \\
&\quad= [[DX,\theta Y],Z] + [[X,\theta DY],Z] + [[X,\theta Y],DZ],
\end{align*}
and (iii) holds.

It is trivial that (iii) implies (iv).
 
Suppose that (iv) holds.
We are going to construct a derivation $\tilde E$ that extends the derivation $E$ of Theorem \ref{thm:construct-der} to the  simple Lie algebra $\Lie{g}$ and preserves heights.

When $X \in \Lie{g}_0 \oplus \Lie{n}$, we set $\tilde E X = EX$.
When $X \in \Lie{g}_0 \oplus \theta \Lie{n}$, we define
\begin{equation}\label{def1}
\tilde E X =  \theta E (\theta X).
\end{equation}
These definitions agree when $X \in \Lie{g}_0$ by part (ii) of Theorem \ref{thm:construct-der}.
It follows from the definition that 
\begin{equation}\label{eq:thetaDtheta}
\theta \tilde E \theta = \tilde E.
\end{equation}

Finally, to show that $\tilde D$ is a derivation, we have to verify that
\begin{equation*}
\tilde D[U, V] = [\tilde DU, V] + [U, \tilde D V]
\qquad\forall U,V \in \Lie{g}.
\end{equation*}
By linearity, it suffices to demonstrate this for $U \in \Lie{g}_h$  and $V \in \Lie{g}_k$, for all possible heights $h$ and $k$.
There are various cases to consider. 
We label the relevant identity $(D_{h,k})$:
\[
\tilde D[U, V] = [\tilde DU, V] + [U, \tilde D V]
\qquad\forall U \in \Lie{g}_h \quad\forall V \in \Lie{g}_k.
\]

\subsubsubsection{Case 1: $h \geq 0$ and $k \geq 0$.}
This case is trivial as $\tilde E$ coincides with $E$ on $\Lie{g}_0 \oplus \Lie{n}$.

\subsubsubsection{Case 2: $h \leq 0$ and $k \leq 0$.}
In this case, we take $X, Y \in \Lie{g}_0 \oplus \theta\Lie{n}$, so $\theta X, \theta Y \in \Lie{g}_0 \oplus \Lie{n}$, and then
\[
\begin{aligned}
\tilde E [X, Y] 
&= \theta E [\theta X, \theta Y] = \theta [E\theta X, \theta Y]  + \theta [\theta X, E\theta Y] \\
&= [\tilde E X, Y] + [X, \tilde EY],
\end{aligned}
\]
and $(D_{h,k})$ holds.

\subsubsubsection{Case 3: $hk < 0$.} 
We need to show that 
\[\tag{$D_{h,k}$}
\tilde E [X, Y ] = [\tilde EX, Y] + [X, \tilde EY]
\qquad\forall X \in \Lie{g}_{h} \quad\forall Y \in \Lie{g}_{k}.
\]
If $h + k \geq 0$, this follows from part (iv) of Theorem \ref{thm:construct-der} and \eqref{eq:thetaDtheta}; otherwise, we conjugate by $\theta$, as in Case 2.

We conclude with the observation that if $\tilde E = \ad(Z)$ for some $Z \in \Lie{m}$ then $\tilde E$ is root space preserving.
\end{proof}

\section{Derivations of $\Lie{n}$}

We are now able to consider a nilpotent Lie algebra $\Lie{n}$ that arises in the Iwasawa decomposition of a real simple Lie algebra $\Lie{g}$.
We write $\Der{n}$ for the space of root space preserving derivations of $\Lie{n}$.

\begin{theorem}\label{thm:4.1}
If $\Lie{g}$ is simple and not isomorphic to $\Lie{so}(n, 1)$ or $\Lie{su}(n,1)$, then every $D$ in $\Der{n}$ is given by
\[
D = \ad(W),
\]
where $W \in \Lie{m}\oplus \Lie{a}$.
\end{theorem}

The main theorem follows from this and Proposition \ref{prop:3.1}.

We prove Theorem \ref{thm:4.1} by showing that every derivation is the sum of a symmetric and a skew-symmetric derivation, and treating these separately.
The symmetric derivations are handled using the following lemma, which reduces matters to showing that symmetric derivations act by scalars on the root spaces.

\begin{lemma}\label{lem:5.2}
 If a derivation $D$ of $\Lie{n}$ acts by a real scalar $\lambda_\alpha$ on each root space $\Lie{g}_\alpha$ where $\alpha \in \Delta$, then $D = \ad(H)$ for some $H \in \Lie{a}$. 
\end{lemma}

\begin{proof}
Since $D$ is a derivation, it is determined by the $\lambda_\alpha$ where $\alpha$ is simple; further, the simple roots form a basis of $\Hom(\Lie{a}, \R)$ and so there exists $H \in \Lie{a}$ such that $\alpha(H) = \lambda_\alpha$ for each simple root.
Hence $D = \ad(H)$.
\end{proof}

\begin{remark}\label{rem:4.2}	
A similar observation is valid when $\Lie{g}$ is complex and $D$ acts by a complex scalar on each root space, since every derivation of a complex Lie algebra is complex linear.
Hence Theorem \ref{thm:4.1} is trivial when $\Lie{g}$ is a split or complex Lie algebra.
In fact, in the split case (that is, when all the roots have multiplicity 1) it follows that $\Der{n} = \ad(\Lie{a})$.
In particular, Theorem \ref{thm:4.1} holds for the algebras with root system $D_n$ (where $n \geq 4$), $E_6$, $E_7$, $E_8$ or $G_2$, since  these are either split or complex, by the classification.
\end{remark}

The skew-symmetric derivations are treated using Theorem \ref{thm:skew-der-char}, which shows that a skew-symmetric derivation $D$ lies in $\ad(\Lie{m})$ when the identity $(E_{\alpha,\beta})$ holds for all simple roots $\alpha$ and $\beta$.
For convenience, we recall this identity:
\[
\tag{$E_{\alpha,\beta}$}
 D[[X, \theta Y], V] = [[DX, \theta Y], V] + [[X, \theta DY], V] + [[X, \theta Y], DV]
\]
for all $X, Y \in \Lie{g}_\alpha$ and all $V \in \Lie{g}_\beta$.

Another key ingredient of our proof, which we use in parallel with the previous obvservation, is a reduction to Lie agebras of rank at most two.
Recall that if $\Epsilon$ is a subset of $\Sigma$, then $\Lie{g}^{\Epsilon}$ denotes the subalgebra of $\Lie{g}$ generated by all the spaces $\Lie{g}_\epsilon$, where $\epsilon \in \Epsilon$.
We also denote by $\Lie{m}^{\Epsilon}$ and $\Lie{n}^{\Epsilon}$ the algebras $\Lie{m} \cap \Lie{g}^{\Epsilon}$ and $\Lie{n} \cap \Lie{g}^{\Epsilon}$, and by $\Sigma^\Epsilon$ the root subsystem $\Sigma \cap \rspan(\Epsilon)$.

Now we come to the proof proper.
Our strategy is to first consider the rank one case; this is known, and we just state what we need. 
Next, we consider the real rank two case, and the third step is to consider the case where the real rank is higher than two.

\subsection{The rank one algebras}

The algebras are well known (see, for instance, Weyl \cite{W}) and the root space preserving derivations are well known. 
We summarise the results in the following proposition for the convenience of the reader.
As the simple algebras for which $\Der{n} \neq \ad(\Lie{m} \oplus \Lie{a})$ are rank one, a case-by-case analysis is appropriate.

\begin{prop}[Riehm \cite{R}, Saal \cite{S}]\label{rango 1}
Let $\Lie{g}$ be simple Lie algebra of real rank one.
Then $\Der{n} = \SymDer{n} \oplus \SkewDer{n}$. 
Moreover,
\begin{enumerate}[(i)]
\item
If $\Lie{g} = \Lie{so}(1,n+1)$, then $\Der{n} = \Lie{sl}(n,\R) \oplus \R$. 
\item
If $\Lie{g} = \Lie{su}(1,n+1)$, then $\Der{n} = \Lie{sp}(n,\R) \oplus \R$.
\item
If $\Lie{g} = \Lie{sp}(1,n+1)$, then $\Der{n} = \Lie{sp}(n-1) \oplus \Lie{sp}(1) \oplus \R$.
\item
If $\Lie{g} = \Lie{f}_{(4,-20)}$, then $\Der{n} = \Lie{so}(7) \oplus \R$.
\end{enumerate}
In all cases, the summand $\R$ corresponds to $\ad(\Lie{a})$.
In the first two cases, $\SymDer{n}$ strictly contains $\ad(\Lie{a})$; in the last two cases, $\SymDer{n}$ coincides with $\ad(\Lie{a})$.  
In all cases,  $\SkewDer{n}$ coincides with $\ad(\Lie{m})$.
\end{prop}

\begin{remark}
In the first two cases, $\Lie{n}$ is not rigid enough to prevent the occurrence of derivations that are not in $\ad(\Lie{m} \oplus \Lie{a})$.
\end{remark}

\begin{proof}
This follows from the work of Riehm \cite{R} and Saal \cite{S}; see also Folland \cite{F} and Pansu \cite{Pu}.
Alternatively, the reader may combine the results about $H$-type algebras with the description of the rank one simple Lie algebras in terms of $H$-type algebras by Cowling, Dooley, Kor\'anyi and Ricci \cite{CDKR2}.
\end{proof}

\begin{cor}\label{cor:Ealphaalpha}
Suppose that $\Lie{g}$ is a simple Lie algebra of arbitrary rank, and $D$ a skew-symmetric root space preserving derivation of $\Lie{n}$.
 Then the identity $(E_{\alpha,\alpha})$ holds for all positive roots $\alpha$.
\end{cor}

\begin{proof}
 For all roots $\alpha$, the restriction $D \rist{\Lie{n}^{\{\alpha\}}}$ is a skew-symmetric root space preserving derivation, and from Theorem \ref{thm:4.1} we deduce that $D \rist{\Lie{n}^{\{\alpha\}}} \in \ad(\Lie{m}^{\{\alpha\}})$.
 Then $(E_{\alpha,\alpha})$ holds by Theorem \ref{thm:skew-der-char}.
\end{proof}

\subsection{The rank two algebras}
Let $\Lie{g}$ be a simple Lie algebra of rank two and denote by $\Lie{n}$ an Iwasawa subalgebra of $\Lie{g}$.
We shall prove that each root space preserving derivation of $\Lie{n}$ is the sum of a symmetric and a skew-symmetric derivation, that the symmetric derivation lies in $\ad(\Lie{a})$, and that the skew-symmetric part satisfies $(E_{\alpha, \beta})$ for all $\alpha$ and $\beta$ in $\Sigma^+$.

Before we analyse the various cases, we need a general result about derivations which we will use when the root system is of type $B_2$ or $BC_2$.

\begin{lemma}\label{lem:4.4a}
Suppose that $\alpha$, $\beta$ and $\alpha+\beta$ are positive roots while $\alpha - \beta$ and $\alpha+2\beta$ are not roots, and that $D \in \Der{n}$.
Then 
\begin{equation}\label{D*-for-B_2}
D\transp  [U,X] = [D\transp  U,X] + [U, D\transp X]
\quad\forall U \in \Lie{g}_\beta \quad\forall X \in \Lie{g}_\alpha.
\end{equation}
\end{lemma}
\begin{proof}
By Proposition \ref{rango 1}, the skew-symmetric part of the restriction of $D$ to $\Lie{g}_\beta$ coincides with $\ad(Z)$ for some $Z$ in $\Lie{m}^{\beta}$. 
Write $D_0$ for $D - \ad(Z)$.
Since $\ad(Z)$ is a skew-symmetric derivation, $D$ satisfies \eqref{D*-for-B_2} if and only if $D_0$ does. 
Thus, by replacing $D$ by $D_0$ if necessary, there is no loss of generality in assuming that the restriction of $D$ to $\Lie{g}_\beta$  is symmetric.

We need to show \eqref{D*-for-B_2}.
The eigenvectors of $D \rist{\Lie{g}_\beta}$ span $\Lie{g}_\beta$, and so by linearity it will suffice to show \eqref{D*-for-B_2} when $U$ is an eigenvector of $D$ and $X$ is arbitrary.
Take $U$ in $\Lie{g}_\beta\setminus\{0\}$ and $\lambda \in \R$ such that $DU = \lambda U$.
By Lemma \ref{lem:UX}, $\ad(U): \Lie{g}_\alpha \to \Lie{g}_{\alpha+\beta}$ is surjective, and so it will suffice to show that
\[
\lip D\transp [U, X] , [U, Y ] \rip = \lip  [D\transp U, X] , [U, Y] \rip  + \lip  [U, D\transp X] , [U, Y] \rip 
\]
for the eigenvector $U$ in $\Lie{g}_\beta$ and arbitrary $X$ and $Y$ in $\Lie{g}_\alpha$.
Now, by the hypothesis that $D$ is a root space preserving derivation, the choice of $U$, \eqref{eq:3.2}, the Jacobi identy, and the fact that $\alpha - \beta$ is not a root, the left hand side is equal to
\begin{align*}
\lip  [U, X] , D[U, Y ] \rip 
&= \lip  [U, X] , [DU, Y ] \rip  + \lip  [U, X] , [U, DY ] \rip  \\
&= \lip  [U, X] , [\lambda U, Y ] \rip  - \lip  X , [\theta U, [U, DY ]] \rip  \\
&= \lip  [\lambda U, X] , [U, Y ] \rip  - \lip  X ,  [[\theta U, U], DY ] \rip - \lip  X ,  [U, [\theta U, DY ]] \rip  \\
&= \lip  [\lambda U, X] , [U, Y ] \rip  - \lip  X ,  \alpha ([\theta U, U]) DY \rip  \\
&= \lip  [D\transp U, X] , [U, Y ] \rip  - \alpha ([\theta U, U])  \lip D\transp X , Y  \rip  \\
\end{align*}
and similarly $\lip  [U, D\transp X] , [U, Y] \rip$ is equal to
\begin{align*}
-\lip  D\transp X , [\theta U, [U, Y ]] \rip 
&= -\lip  D\transp X ,  [[\theta U,U], Y ] \rip -  \lip  D\transp X , [U, [\theta U, Y ]] \rip  \\
&= -\lip  D\transp X ,  \alpha([\theta U,U]) Y \rip   \\
&= - \alpha ([\theta U, U])  \lip  D\transp X , Y  \rip  .
\end{align*}
The result now follows.
\end{proof}

\subsubsection{The case $A_2$.}
Until further notice, we assume that $\Lie{g}$ has root system $A_2$, the simplest indecomposible root system of rank $2$.
We label the simple roots $\alpha$ and $\beta$, so that the highest root is $\alpha+\beta$ and $\Sigma^+ = \{ \alpha, \beta, \alpha+\beta\}$.
With the notation of Lemma \ref{ciatti-htype}, $\Sigma_1 = \{\alpha, \beta\}$ and $\Sigma_0 = \varnothing$.
We shall use the result of Ciatti \cite[Proposition 4.1]{C3} about the structure of $\Lie{n}$, giving a proof for the convenience of the reader.

\begin{lemma}\label{A2"}
For every nontrivial $X$ in $\Lie{g}_\alpha$,
\begin{equation}\label{A2}
\Lie{g}_{\beta}=\{J_Z X : Z\in \Lie{g}_{\alpha+\beta} \}
\end{equation}
and
\begin{equation}\label{A2'}
\Lie{g}_{\alpha}=\{J_{Z'}J_Z X : Z,Z'\in \Lie{g}_{\alpha+\beta} \} .
\end{equation}
\end{lemma}

\begin{proof}
We may and shall assume that $X$ is a unit vector, and take $Y \in \Lie{g}_\beta$.
By \eqref{eq:3.10} and the Jacobi identity, 
\[
   J_{[Y,X]}X= [[Y,X], \theta X]
= [[\theta X,X], Y]
= \beta([\theta X, X]) Y
= Y,
\]
which proves that $Y \in \{J_Z X : Z \in \Lie{g}_{\alpha+\beta}\} $.

Now, by \eqref{eq:3.11}, $J_Z$ is a linear isomorphism that exchanges $\Lie{g}_{\beta}$ and $\Lie{g}_{\alpha}$ for all nonzero $Z$ in $\Lie{g}_{\alpha+\beta}$, so \eqref{A2'} follows from \eqref{A2}.
\end{proof}

\begin{prop}\label{prop:A2-sum} 
Every root space preserving derivation of $\Lie{n}$ is the sum of a symmetric and a skew-symmetric derivation.
\end{prop}

\begin{proof}
By Lemma \ref{ciatti-htype},  $\Lie{n}$ is $H$-type.
The result follows from Corollary~\ref{cor:2.5b}.
\end{proof}

\begin{prop}\label{prop:A2-sym} 
Every symmetric root space preserving derivation $D$ of $\Lie{n}$ lies in $\ad( \Lie{a})$.
\end{prop}

\begin{proof}
From Corollary~\ref{cor:2.8}, $D$ is the sum of a symmetric derivation $D_0$ that vanishes on $\Lie{g}_{\alpha+\beta}$ and $\ad(H)$ for some $H$ in $\Lie{a}$.
Since $D_0$ is symmetric and preserves root spaces, we may take an eigenvector $X$ of $D_0$ in $\Lie{g}_\alpha$ with corresponding eigenvalue $\lambda$.
By Proposition \ref{prop:2.1} and Lemma \ref{A2"}, $D_0$ anticommutes with the maps $J_Z$ and so acts as $-\lambda$ on $\Lie{g}_\beta$, and hence as $\lambda$ on $\Lie{g}_\alpha$. 
This implies that $D_0$ acts as $\lambda$ on $\Lie{g}_\alpha$ and $-\lambda$ on $\Lie{g}_\beta$, so $D_0$ lies in $\ad(\Lie{a})$ by Lemma \ref{lem:5.2}.
\end{proof}

\begin{prop}\label{prop:A2-skew} 
The basic derivation identity $(E_{\gamma,\delta})$ holds as $\gamma$ and $\delta$ range over the set $\{\alpha, \beta\}$ of simple roots.
Consequently, every derivation $D$ in $\SkewDer{n}$ is equal to $\ad(Z)$ for some $Z$ in $\Lie{m}$.
\end{prop}

\begin{proof}
Recall the basic derivation identity $(E_{\gamma,\delta})$, that is, the identity
\[
D[[X,\theta Y],Z] = [[DX,\theta Y, Z] + [[X,\theta DY],Z] + [[X,\theta Y],DZ]
\]
for all $X, Y \in \Lie{g}_\gamma$ and all $Z \in \Lie{g}_\delta$.
By Theorem \ref{thm:skew-der-char}, it suffices to prove $(E_{\gamma,\delta})$ as $\gamma$ and $\delta$ range over $\{\alpha, \beta\}$.
The identities $(E_{\alpha,\alpha})$ and $(E_{\beta,\beta})$ hold by Corollary \ref{cor:Ealphaalpha}.

Now we prove $(E_{\alpha,\beta})$. 
Suppose that $X, Y  \in \Lie{g}_\alpha$ and $Z \in \Lie{g}_\beta$.
Since $\beta - \alpha$ is not a root,
\begin{align*}
&D[[X,\theta Y],Z] - [[DX,\theta Y], Z] - [[X,\theta DY],Z] - [[X,\theta Y],DZ] \\
&\quad= D[[X,Z], \theta Y] - [[DX, Z], \theta Y] - [[X,Z], \theta DY] - [[X, DZ] ,\theta Y] \\ 
&\quad= D[[X,Z], \theta Y] - [D[X, Z], \theta Y] - [[X,Z], \theta DY]  \\ 
&\quad= D[W, \theta Y] - [DW, \theta Y] - [W, \theta DY]  ,
\end{align*}
where $W =  [X,Z] \in \Lie{g}_{\alpha+\beta}$; it will suffice to prove that this is $0$ for all $W \in \Lie{g}_{\alpha+\beta}$ and all $Y \in \Lie{g}_\alpha$.
By the definition of $J_W$, for $W \in \Lie{g}_{\alpha+\beta}$, we may rewrite the last expression as
\[
DJ_W Y - J_{DW} Y - J_W DY,
\]
and since $D$ is a skew-symmetric derivation, this is $0$ by \eqref{eq:2.4}.

We exchange the roles of $\alpha$ and $\beta$ to prove the remaining identity.
\end{proof}

\subsubsection{The case $B_2$.}
Until further notice, we assume that $\Lie{g}$ has root system $B_2$.
We denote by $\alpha$ and $\beta$ the simple roots, with $\beta$ longer than $\alpha$.
Hence $\omega =2\alpha+\beta$ and $\Sigma^+ = \{\alpha, \beta, \alpha+\beta, 2\alpha+\beta\}$.

The first proposition is the basic result: it establishes that every element of $\Der{n}$ is a sum of a symmetric and a skew-symmetric derivation.

\begin{prop}\label{pro:B2-sum}  
If $D$ is in $\Der{n}$, then its transpose $D\transp $ is also in $\Der{n}$.
\end{prop}

\begin{proof}
By Lemma \ref{ciatti-htype}, the algebra $\Lie{g}_\alpha \oplus \Lie{g}_{\alpha+\beta} \oplus \Lie{g}_{2\alpha+\beta}$ is $H$-type.
By Corollary \ref{cor:2.5b}, the restriction of $D\transp $ to this $H$-type algebra is a derivation.
Hence it suffices to show that
\begin{equation*}
D\transp [U,X] = [D\transp U, X] + [U, D\transp X]
\end{equation*}
for all $X$ in $\Lie{g}_\alpha$ and all $U$ in $\Lie{g}_\beta$.
The proposition now follows from Lemma \ref{lem:4.4a}.
\end{proof}

Now we describe the symmetric derivations. 

\begin{prop}\label{prop:B2-sym} 
Every derivation $D$ in $\SymDer{n}$ is equal to $\ad(H)$ for some $H$ in $\Lie{a}$.
\end{prop}

\begin{proof}
By Lemma \ref{lemH}, $\Lie{g}_\alpha \oplus \Lie{g}_{\alpha+\beta} \oplus \Lie{g}_{2\alpha+\beta}$ is an $H$-type algebra.
In light of Corollary~\ref{cor:2.8}, we may assume that $D$ vanishes on $\Lie{g}_{2\alpha+\beta}$, by subtracting $\ad(H)$ for a suitable $H$ in $\Lie{a}$.

The derivation $D$, being symmetric, may be diagonalized with real eigenvalues.
We fix eigenvectors $U$ in $\Lie{g}_\beta$ with eigenvalue $\lambda$ and $X$ in $\Lie{g}_\alpha$ with eigenvalue $\mu$.
Since $D$ is a derivation, 
\begin{equation*}
D[[U,X],X]=(\lambda+2\mu)[[U,X],X].
\end{equation*}

Now $\alpha$ and $\beta$ satisfy the hypotheses of Lemma~\ref{lem:UXX}, and so  $[[U,X],X]$ is nonzero.
Since $D$ vanishes on $\Lie{g}_{2\alpha+\beta}$,
\begin{equation*}
\lambda+2\mu = 0.
\end{equation*}

We vary the eigenvector $U$, holding $X$ fixed: this shows that $\lambda$ is independent of $X$.
Similarly, $\mu$ is independent of $U$.
By Lemma \ref{lem:5.2}, this implies the proposition.
\end{proof}

\begin{prop}\label{prop:B2-skew} 
The basic derivation identity $(E_{\gamma,\delta})$ holds as $\gamma$ and $\delta$ range over the set $\{\alpha, \beta\}$ of simple roots.
Consequently, every derivation $D$ in $\SkewDer{n}$ is equal to $\ad(Z)$ for some $Z$ in $\Lie{m}$.
\end{prop}

\begin{proof}
Recall the basic derivation identity $(E_{\gamma,\delta})$, that is, the identity
\[
D[[X,\theta Y],Z] = [[DX,\theta Y, Z] + [[X,\theta DY],Z] + [[X,\theta Y],DZ]
\]
for all $X, Y \in \Lie{g}_\gamma$ and all $Z \in \Lie{g}_\delta$.
Again by Theorem \ref{thm:skew-der-char}, we need to prove $(E_{\gamma,\delta})$ as $\gamma$ and $\delta$ range over $\{\alpha, \beta\}$.
The identities $(E_{\alpha,\alpha})$ and $(E_{\beta,\beta})$ hold by Corollary \ref{cor:Ealphaalpha}.

Now we prove $(E_{\beta, \alpha})$. 
Suppose that $X, Y  \in \Lie{g}_\beta$ and $Z \in \Lie{g}_\alpha$.
Since $\beta - \alpha$ is not a root,
\begin{align*}
&D[[X,\theta Y],Z] - [[DX,\theta Y], Z] - [[X,\theta DY],Z] - [[X,\theta Y],DZ] \\
&\quad= D[[X,Z], \theta Y] - [[DX, Z], \theta Y] - [[X,Z], \theta DY] - [[X, DZ] ,\theta Y] \\ 
&\quad= D[[X,Z], \theta Y] - [D[X, Z], \theta Y] - [[X,Z], \theta DY]  \\ 
&\quad= D[W, \theta Y] - [DW, \theta Y] - [W, \theta DY]  ,
\end{align*}
where $W =  [X,Z] \in \Lie{g}_{\alpha+\beta}$; it will suffice to prove that this is $0$ for all $W \in \Lie{g}_{\alpha+\beta}$ and all $Y \in \Lie{g}_\beta$.
By Lemma \ref{lem:UX}, $\ad(Y)$ maps $\Lie{g}_\alpha$ onto $\Lie{g}_{\alpha+\beta}$, so it will suffice to prove that
\[
D[[U,Y], \theta Y] - [D[U,Y], \theta Y] - [[U,Y], \theta DY] =0
\]
for all $U \in \Lie{g}_\alpha$ and all $Y \in \Lie{g}_\beta$.
Since $\alpha - \beta$ is not a root,  $[[R,S], \theta T] = [R,[S, \theta T]]$ for all $R \in \Lie{g}_\alpha$ and all $S,T \in \Lie{g}_\beta$, whence
\[
\begin{aligned}
&D[[U,Y], \theta Y] - [D[U,Y], \theta Y] - [[U,Y], \theta DY] \\
&\quad=D[[U,Y], \theta Y] - [[DU,Y], \theta Y] -  [[U,DY], \theta Y] - [[U,Y], \theta DY] \\
&\quad=D[U,[Y, \theta Y]] - [DU,[Y, \theta Y]] -  [U,[DY, \theta Y]] - [U,[Y, \theta DY]] \\
&\quad =     \lip \alpha, \beta \rip \|Y\|^2 DU - \lip \alpha, \beta \rip \|Y\|^2 DU 
 -  [U,[DY, \theta Y]] - [U,[Y, \theta DY]] \\
&\quad =  [ [DY, \theta Y] + [Y, \theta DY]  , U].
\end{aligned}
\]
Now if $X \perp Y$, then $[X, \theta Y] \in \Lie{m}$ and hence 
\[
[X, \theta Y] + [Y, \theta X] = \theta [X, \theta Y] + [Y, \theta X] =  [\theta X, Y] + [Y, \theta X]= 0. 
\]
Applying this with $X$ equal to $DY$ finishes the proof of $(E_{\beta, \alpha})$.

It remains to prove $(E_{\alpha,\beta})$.
Take $X, Y \in \Lie{g}_\alpha$ and $U,Z \in \Lie{g}_\beta$.
Then
\begin{align*}
&\lip D[[X, \theta Y],Z] - [[DX, \theta Y],Z]  - [[X, \theta DY],Z]  - [[X, \theta Y],DZ], U \rip \\
&\quad= -\lip [[X, \theta Y],Z], DU \rip + \lip [DX, \theta Y], [U, \theta Z] \rip  \\
&\qquad  + \lip [X, \theta DY], [U, \theta Z]  \rip  + \lip [X, \theta Y], [U, \theta DZ]  \rip \\
&\quad= \lip [X, \theta Y], [DU,\theta Z] \rip - \lip DX,  [[U, \theta Z], Y] \rip  \\
&\qquad- \lip X, [[U, \theta Z],DY]  \rip  - \lip X, [[U, \theta DZ] ,Y] \rip \\
&\quad=-\lip X, [[DU,\theta Z],Y] \rip + \lip X,  D[[U, \theta Z], Y] \rip  \\
&\qquad - \lip X, [[U, \theta Z],DY]  \rip  - \lip X, [[U, \theta DZ] ,Y] \rip \\
&\quad=  \lip X, D[[U, \theta Z], Y]  - [[DU,\theta Z],Y]    - [[U, \theta Z],DY]  - [[U, \theta DZ] ,Y] \rip.
\end{align*}
This shows that $(E_{\alpha,\beta})$ and $(E_{\beta,\alpha})$ are equivalent, so we are done.

Note that we have not used the fact that $2\alpha$ and $2(\alpha+\beta)$ are not roots, so this argument holds in the $BC_2$ case too.
\end{proof}

This completes our discussion of the algebras with root system $B_2$.
We remind the reader that $C_2$ is the same as $B_2$.
The algebras with root system $G_2$ are covered by Remark \ref{rem:4.2}.
It remains to consider the algebras with root system $BC_2$.

\subsubsection{The case $BC_2$.}
Until further notice, we assume that $\Lie{g}$ has root system $BC_2$.
Denote by $\alpha$ and $\beta$ the simple roots, with $\alpha$ orthogonal to the highest root $\omega$. 
Then $\Sigma^+ = \{\alpha, 2\alpha, \beta, \alpha+\beta, 2\alpha + \beta, 2\alpha + 2\beta\}$ and $\omega = 2\alpha + 2\beta$.

Note that $\{ \pm 2 \alpha, \pm \beta, \pm (2\alpha+\beta), \pm (2\alpha+2\beta)\}$ is a root subsystem of type $B_2$, write $\Lie{n}_{\mathrm{sub}}$  for $\Lie{g}_{\beta} \oplus \Lie{g}_{2\alpha} \oplus \Lie{g}_{2\alpha + \beta} \oplus \Lie{g}_{2\alpha + 2\beta}$.
The results  of the previous subsection apply to the root space preserving derivations of the subalgebra $\Lie{n}_{\mathrm{sub}}$ to give us information about derivations of $\Lie{n}$.

The first step is to establish the analogue of Proposition \ref{pro:B2-sum}.

\begin{prop}\label{prop:BC2-sum}
If $D$ is in $\Der{n}$, then its transpose $D\transp $ is also in $\Der{n}$.
\end{prop}

\begin{proof}
By linearity, it suffices to show that 
\begin{equation}\label{eq:der-gamma-delta}
D\transp [X,Y] = [D\transp X, Y] + [X, D\transp Y] 
\quad\forall X \in \Lie{g}_\gamma \quad\forall Y \in \Lie{g}_\delta, 
\end{equation}
as $\gamma$ and $\delta$ range over $\Sigma^+$.
As $D$ and hence also $D\transp $ preserve root spaces, this is trivial unless $\gamma+\delta$ is a root.
Moreover, by Corollary \ref{cor:2.5b}, the restrictions of $D\transp $ to  the $H$-type algebras $\Lie{g}_{\beta} \oplus \Lie{g}_{\alpha+ \beta} \oplus \Lie{g}_{2\alpha + \beta} \oplus \Lie{g}_{2\alpha + 2\beta}$ and $\Lie{g}_\alpha \oplus \Lie{g}_{2\alpha}$ are derivations, and by Proposition \ref{pro:B2-sum}, the restriction of $D\transp $ to  $\Lie{g}_{\beta} \oplus \Lie{g}_{2\alpha} \oplus \Lie{g}_{2\alpha + \beta} \oplus \Lie{g}_{2\alpha + 2\beta}$ is a derivation.

Thus it suffices to prove \eqref{eq:der-gamma-delta} when $(\gamma,\delta)$ is either $(\alpha,\beta)$ or $(\alpha,\alpha + \beta)$.
Lemma \ref{lem:4.4a} takes care of the case when $\gamma = \alpha$ and $\delta=\beta$.

Since $2(2\alpha+\beta)$ is not a root, Proposition \ref{rango 1} implies that there exists $Z$ in $\Lie{m}^{2\alpha+\beta}$ that agrees with the skew-symmetric part of $D$ on  $\Lie{g}_{2\alpha+\beta}$.
By subtracting $\ad(Z)$ from $D$ if necessary, we may suppose that $D$ is symmetric on $\Lie{g}_{2\alpha+\beta}$.
Now Corollary \ref{cor:3.7} gives \eqref{eq:der-gamma-delta}.
\end{proof}

Once again, we consider the symmetric derivations.

\begin{prop}\label{pro:BC2-sym} 
If the root system of $\Lie{g}$ is $BC_2$, then every derivation in $\SymDer{n}$ is given by $\ad(H)$ for some $H$ in $\Lie{a}$.
\end{prop}

\begin{proof}
By Proposition \ref{prop:B2-sym}, the restriction of $D$ to $\Lie{n}_{\mathrm{sub}}$ is given by $\ad(H)$ for some $H$ in $\Lie{a}$.
By subtracting $\ad(H)$ if necessary we may suppose that $D$ vanishes on $\Lie{n}_{\mathrm{sub}}$; it will then suffice to show that $D$ is trivial.

To do this, we pick an eigenvector $X$ of $D$ in $\Lie{g}_{\alpha}$ with eigenvalue $\lambda$ and $U \in \Lie{g}_{\beta} \setminus\{0\}$.
Since $D$ is a derivation and $DU = 0$
\[
D[[U,X],X] = 2 \lambda [[U,X],X].
\]
However, $D[[U,X],X]=0$, since $[[U,X],X]$ lies in $\Lie{g}_{2\alpha+\beta} \subset \Lie{n}_{\mathrm{sub}}$.
Since $[[U,X],X] \neq 0$ by Lemma \ref{lem:UXX},  $\lambda=0$, and $D$ is trivial on $\Lie{g}_\alpha$.
Since $D$ is also trivial on $\Lie{g}_\beta$, it is trivial on $\Lie{g}_{\alpha+\beta}$, and hence trivial on all the root spaces.
\end{proof}

We conclude our discussion of the rank $2$ case with a description of the skew-symmetric derivations.

\begin{prop}\label{prop:BC2-skew}
 The basic derivation identity $(E_{\gamma,\delta})$ holds as $\gamma$ and $\delta$ range over the set $\{\alpha, \beta\}$ of simple roots.
Consequently, every derivation $D$ in $\SkewDer{n}$ is equal to $\ad(Z)$ for some $Z$ in $\Lie{m}$.
\end{prop}

\begin{proof}
This follows from Proposition \ref{prop:B2-skew}, which also holds in the root system $BC_2$, and Theorem \ref{thm:skew-der-char}.
\end{proof}

\subsection{The general case}
Now we prove Theorem \ref{thm:4.1}.
Henceforth, $\Lie{g}$ denotes a real simple Lie algebra of rank at least $3$, and $\Lie{n}$ is an Iwasawa nilpotent subalgebra of $\Lie{g}$.

\begin{prop}\label{prop:general-sum}
 Suppose that $D$ is a derivation of $\Lie{n}$.
 Then $D\transp $ is also a derivation of $\Lie{n}$.
\end{prop}

\begin{proof}
By linearity, this follows provided that 
\[
D\transp [X,Y] = [D\transp X,Y] + [X,D\transp Y] 
\]
for all $X \in \Lie{g}_\gamma$ and all $Y \in \Lie{g}_\delta$ where $\gamma$ and $\delta$ range over $\Sigma^+$.
This is obvious if $\gamma+\delta$ is not a root, while if $\gamma+\delta$ is a root, then it follows by restricting $D$ to $\Lie{n}^{\{\gamma,\delta\}}$.
\end{proof}

\begin{prop}\label{prop:general-sym}
 Suppose that $D$ is a symmetric derivation of $\Lie{n}$.
 Then $D$ lies in $\ad(\Lie{a})$.
\end{prop}

\begin{proof}
Again, by restricting to rank two subalgebras, we may show that $D$ acts as a scalar on each root space. 
By Lemma \ref{lem:5.2}, $D \in \ad(\Lie{a})$.
\end{proof}

\begin{prop}\label{prop:general-skew}
 Suppose that $D$ is a skew-symmetric derivation of $\Lie{n}$.
 Then $D$ lies in $\ad(\Lie{m})$.
\end{prop}

\begin{proof}
Let $D$ be a skew-symmetric root space preserving derivation of $\Lie{n}$.
Again, by restricting to rank two subalgebras, we may show that $D$ satisfies the basic derivation identity $(E_{\gamma,\delta})$ whenever $\gamma$ and $\delta$ are positive roots.
Hence $D \in \ad(\Lie{m})$ by Theorem \ref{thm:skew-der-char}.
\end{proof}

\end{document}